\documentclass[12pt, reqno, a4paper]{amsart}

\usepackage{amssymb}
\usepackage{hyperref}
\usepackage{bbm}

\usepackage[english]{babel}

\usepackage[matrix, arrow, curve]{xy}

\theoremstyle{plain}
\newtheorem{theorem}{Theorem}[section]
\newtheorem{lemma}[theorem]{Lemma}
\newtheorem{proposition}[theorem]{Proposition}
\newtheorem{corollary}[theorem]{Corollary}
\newtheorem*{conjecture}{Conjecture}

\theoremstyle{remark}
\newtheorem{remark}[theorem]{Remark}

\theoremstyle{definition}
\newtheorem{example}[theorem]{Example}

\newtheorem{definition}[theorem]{Definition}

\numberwithin{equation}{section}

\DeclareMathOperator{\Id}{Id}
\DeclareMathOperator{\UT}{UT}

\DeclareMathOperator{\id}{id}
\newcommand{\ch}{\mathop\mathrm{char}}
\DeclareMathOperator{\im}{im}

\DeclareMathOperator{\GL}{GL}

\DeclareMathOperator{\End}{End}
\DeclareMathOperator{\Aut}{Aut}

\DeclareMathOperator{\supp}{supp}

\DeclareMathOperator{\Ann}{Ann}

\DeclareMathOperator{\PIexp}{PIexp}

\DeclareRobustCommand{\No}{\ifmmode{\nfss@text{\textnumero}}\else\textnumero\fi} 

\makeatletter
\@namedef{subjclassname@2020}{\textup{2020} Mathematics Subject Classification}
\makeatother

\begin{document}

\title{Graded group actions and generalized $H$-actions compatible with gradings}

\author{A.\,S. Gordienko}
\address{Department of Higher Algebra,
Faculty of Mechanics and  Mathematics,
M.\,V.~Lomonosov Moscow State University,
Leninskiye Gory, d.1,  Main Building, GSP-1, 119991 Moskva, Russia }

\email{alexey.gordienko@math.msu.ru}


\begin{abstract} 
We introduce the notion of a graded group action on a graded algebra or, which is the same, a group action by graded pseudoautomorphisms. An algebra with such an action is a natural generalization of an algebra with a super- or a pseudoinvolution.
We study groups of graded pseudoautomorphisms, show that the Jacobson radical of a group graded finite dimensional associative algebra $A$ over a field of characteristic $0$ is stable under graded pseudoautomorphisms, prove the invariant version of the Wedderburn--Artin Theorem and the analog of Amitsur's conjecture for the codimension growth of graded polynomial $G$-identities in such algebras $A$ with a graded action of a group $G$.
\end{abstract}


\thanks{Partially supported by the Russian Science Foundation grant \No\,22-11-00052.
}

\maketitle

Classifying all varieties of algebras with respect to their polynomial identities seems to be an extremely hard problem since not too many examples of algebras are known where a basis for polynomial identities has been explicitly found~\cite{Bahturin, DrenKurs, ZaiGia}. One of the possible ways to overcome this difficulty is to classify varieties with respect to numeric sequences related to their polynomial identities.
One of such sequences is the sequence of codimensions. 
By the definition, $c_n(A) := \dim \frac{W_n}{W_n\cap \Id(A)}$ where $W_n$ is the vector space of multilinear
polynomials in the non-commuting variables $x_1,\ldots, x_n$ and $\Id(A)$ is the set of polynomial identities of $A$, i.e. those polynomials that vanish under all evaluations in $A$. Codimensions arise naturally e.g. when one calculates a basis for polynomial identities of an algebra over a field of characteristic~$0$. In addition, they are a useful tool to prove that an algebra satisfies a nontrivial polynomial identity~\cite{RegTensor}. Finally, there is a tight relationship between the asymptotic
behaviour of codimensions and the structure of the algebra~\cite[Chapter 6]{ZaiGia}.
 Another important sequence is the sequence of colengths. Colengths appear when one considers the natural action of the $n$th symmetric group $S_n$ on $\frac{W_n}{W_n\cap \Id(A)}$. Colengths provide an essential instrument for proving upper bounds on codimensions~\cite[Proposition 6.2.6]{ZaiGia}.

In 1980's, S.\,A.~Amitsur made a conjecture on the asymptotic behavior of codimensions in associative algebras:

\begin{conjecture}[S.\,A.~Amitsur]
	Let $A$ be an associative p.i. algebra (=algebra satisfying a polynomial identity that holds not in all associative algebras)
	over a field of characteristic $0$. Then there exists $\PIexp(A):=\lim\limits_{n\to\infty} \sqrt[n]{c_n(A)} \in \mathbb Z_+$.
\end{conjecture} 

The number $\PIexp(A)$ is called the \textit{exponent of the codimension growth of polynomial identities} or the \textit{PI-exponent} of $A$.

Amitsur's conjecture was proved in 1999 by
A.~Giambruno and M.V.~Zaicev~\cite[Theorem~6.5.2]{ZaiGia}. Analogs of Amitsur's conjectures were proved in 2002 by M.V.~Zaicev~\cite{ZaiLie}
for finite dimensional Lie algebras and in 2011 by A.~Giambruno,
I.P.~Shestakov, M.V. Zaicev~\cite{GiaSheZai} for finite dimensional Jordan and alternative
algebras. Moreover, formulas that relate $\PIexp(A)$ with the structure of $A$ were found. The Wedderburn--Mal'cev Theorem, as well as its non-associative analogs, was an important tool in the proofs.

Alongside with ordinary polynomial
identities of algebras, graded polynomial identities, $G$- and
$H$-identities are
important too~\cite{BahtGiaZai, BahturinLinchenko, BahtZaiGraded, BahtZaiGradedExp, BahtZaiSehgal,
	BereleHopf, Linchenko}.
Usually, to describe such identities is easier
than to describe the ordinary ones. Furthermore, the graded polynomial identities, $G$- and
$H$-identities completely determine the ordinary polynomial identities. 
Therefore the question arises whether the conjecture
holds for graded codimensions, $G$- and $H$-codimensions.
The analog of Amitsur's conjecture
for codimensions of graded identities was proved in 2010--2011 by
E.~Aljadeff,  A.~Giambruno, and D.~La~Mattina~\cite{AljaGia, AljaGiaLa, GiaLa}
for all associative PI-algebras graded by a finite group.

In 2013, the analog of the conjecture
was proved~\cite{ASGordienko3} for finite dimensional associative algebras with an action
of a semisimple Hopf algebra. This result was obtained as a consequence of~\cite[Theorem~5]{ASGordienko3}, where the author considered finite dimensional associative algebras with a generalized $H$-action of an associative algebra $H$ with $1$.
In the proof, he required the existence of an $H$-invariant Wedderburn--Mal'cev and Wedderburn--Artin decompositions. The first restriction was removed in~\cite{ASGordienko8}.  In~\cite{ASGordienko15} the author provided sufficient conditions for the conjecture to hold in algebras with a not necessarily $H$-invariant Jacobson radical.

Recently, studies of asymptotic behaviour of the corresponding types of polynomial
identities in finite dimensional associative algebras with super- and pseudoinvolutions
have gained popularity~\cite{AljaGiaKar, CentEstradaIop,
	dosSantos,
GiaIopLaMa1, GiaIopLaMa2, Ioppolo, IopMar}. R.\,B. dos Santos~\cite{dosSantos} was the first to notice that
an algebra with a superinvolution is an algebra with a generalized $H$-action.
In Theorem~\ref{TheoremGradGenActionReplace} 
below we show that the class of
additional structures on algebras that can be reduced to generalized $H$-actions
includes not only superinvolutions and pseudoinvolutions, but also any actions on an algebra compatible with a grading having a finite support.

In order to capture the properties of a $G$-action for a group $G$ that still force the Jacobson radical to be $G$-invariant and make it possible to prove the analog of Amitsur's conjecture for graded polynomial $G$-identities and the analog of the Wedderburn--Artin Theorem, we introduce the notions of a graded pseudoautomorphism of a graded algebra (see Definition~\ref{DefGrPseudoAuto}) and a graded group action (see
Definition~\ref{DefGradedAction}). The latter is precisely a group action by graded pseudoautomorphisms.

In Theorems~\ref{Theorem2x2PseudoTrivial}, \ref{TheoremPseudoGrassmann} and \ref{TheoremPseudoAutoImageTauIsQNilpAlg} we calculate the groups of graded pseudoautomorphisms for several algebras.
In Theorem~\ref{TheoremTGradedGActionInvRad} we show that the Jacobson radical is invariant under graded pseudoautomorphisms. In Theorem~\ref{TheoremTGradedGActionInvWedderburnArtin} we prove the invariant version of the Wedderburn--Artin Theorem. In Theorem~\ref{TheoremTGrGActionAssoc} we show that the analog of Amitsur's conjecture holds for graded polynomial $G$-identities of finite dimensional associative algebras over a field of characteristic $0$ with a graded action of an arbitrary group $G$,
thus generalizing the results from~\cite{dosSantos,
	GiaIopLaMa1, GiaIopLaMa2, Ioppolo, IopMar}. 

\section{Gradings and graded polynomial identities}

Let $A$ be a (not necessarily associative) algebra over a field $\mathbbm{k}$ and let $T$ be a set.
We say that a decomposition $\Gamma \colon A=\bigoplus_{t\in T} A^{(t)}$ is a \textit{$T$-grading} on
the algebra $A$ if for every $s,t \in T$ there exists $r\in T$
such that $A^{(s)}A^{(t)}\subseteq A^{(r)}$. In this case we call $A$
a \textit{$T$-graded} algebra.
If $T$ is a group and $A^{(s)}A^{(t)}\subseteq A^{(st)}$
for every $s,t \in T$, then we say that $\Gamma$ is a \textit{group $T$-grading}.

The set $\supp \Gamma := \lbrace t\in T \mid A^{(t)}\ne 0\rbrace$ is called 
the \textit{support} of $\Gamma$.

Consider the absolutely free non-associative algebra $\mathbbm{k}\lbrace X^{T\text{-}\mathrm{gr}} \rbrace$
on the disjoint union $$X^{T\text{-}\mathrm{gr}}:=\bigsqcup_{t \in T}X^{(t)}$$
of the sets $X^{(t)} = \{ x^{(t)}_1,
x^{(t)}_2, \dots \}$, i.e. the algebra of all non-associative polynomials in variables from $X^{T\text{-}\mathrm{gr}}$.

We say that $f=f\left(x^{(t_1)}_{i_1}, \dots, x^{(t_s)}_{i_s}\right)$ is
a \textit{graded polynomial identity} for
a $T$-graded algebra $A=\bigoplus_{t\in T}
A^{(t)}$
and write $f\equiv 0$
if $f\left(a^{(t_1)}_1, \dots, a^{(t_s)}_s\right)=0$
for all $a^{(t_j)}_j \in A^{(t_j)}$, $1 \leqslant j \leqslant s$.
The set $\Id^{T\text{-}\mathrm{gr}}(A)$ of graded polynomial identities of
$A$ is
an ideal of $\mathbbm{k}\lbrace
X^{T\text{-}\mathrm{gr}}\rbrace$.

\begin{example}\label{ExampleIdGr}
	Consider the $\mathbb{Z}/2\mathbb{Z}$-grading $\UT_2(\mathbbm{k})=\UT_2(\mathbbm{k})^{(0)}\oplus \UT_2(\mathbbm{k})^{(1)}$
	on the algebra $\UT_2(\mathbbm{k})$ of upper triangular $2\times 2$ matrices over a field $\mathbbm{k}$
	defined by
	$\UT_2(\mathbbm{k})^{(0)}=\left(
	\begin{array}{cc}
	\mathbbm{k} & 0 \\
	0 & \mathbbm{k}
	\end{array}
	\right)$ and $\UT_2(\mathbbm{k})^{(1)}=\left(
	\begin{array}{cc}
	0 & \mathbbm{k} \\
	0 & 0
	\end{array}
	\right)$. We have $$\left[x^{(0)}, y^{(0)}\right]:=x^{(0)} y^{(0)} - y^{(0)} x^{(0)}
	\in \Id^{\mathbb{Z}/2\mathbb{Z}\text{-}\mathrm{gr}}(\UT_2(\mathbbm{k}))$$
	and $x^{(1)} y^{(1)} 
	\in \Id^{\mathbb{Z}/2\mathbb{Z}\text{-}\mathrm{gr}}(\UT_2(\mathbbm{k}))$.
\end{example}

Every ungraded algebra $A$ can be graded by the trivial group $\lbrace 1\rbrace$. In this case graded polynomial identities
are just ordinary \textit{polynomial identities} of $A$.

\section{Graded pseudoautomorphisms of graded algebras}

In this section we introduce graded pseudoautomorphisms. Our motivation comes from Examples~\ref{ExampleSuperInvolution}
and~\ref{ExamplePseudoInvolution} below.

\begin{definition}\label{DefGrPseudoAuto}
	Let $\Gamma \colon A=\bigoplus_{t\in T} A^{(t)}$ be a $T$-grading on a
	 (not necessarily associative) algebra $A$ over a field $\mathbbm{k}$ for some set $T$.
We say that a linear bijection $\varphi \colon A \to A$ is a
\textit{pseudoautomorphism} of $\Gamma$
or a \textit{graded pseudoautomorphism} of $A$
if $\varphi \left( A^{(t)} \right) \subseteq A^{(t)}$ for all $t\in T$
and for every $s,t \in T$ there exist $\alpha(s,t), \beta(s,t) \in \mathbbm{k}$
such that
\begin{equation}\label{EqGradedPseudoAutomorphism}
\varphi(ab)=\alpha(s,t)\,\varphi(a)\varphi(b)+\beta(s,t)\,\varphi(b)\varphi(a)
\end{equation}
  for all $s,t\in
   T$, $a\in A^{(s)}$ and $b\in A^{(t)}$.
\end{definition} 
  
  If \eqref{EqGradedPseudoAutomorphism} holds for $\alpha(s,t)=1$,
  $\beta(s,t)=0$ for all $s,t\in T$, then we say that $\varphi$ is an
\textit{automorphism} of $\Gamma$
or a \textit{graded automorphism} of $A$. In other words, $\varphi$ is a graded automorphism
of $A$ if and only if $\varphi$ is an automorphism that preserves the graded components.

Below we give some other important examples of graded pseudoautomorphisms:

\begin{example}\label{ExampleSuperInvolution} Let $A=A^{(0)}\oplus A^{(1)}$ be a $\mathbb Z/2\mathbb Z$-graded algebra. A linear map $\star \colon A \to A$ is a \textit{superinvolution} (see \cite{RacineSuperInvolution}) if
	$\left(A^{(k)}\right)^\star = A^{(k)}$ for all $k=0,1$, $a^{\star\star}=a$ for all $a\in A$,
	$(ab)^\star=(-1)^{k\ell}b^\star a^\star$ for all $a\in A^{(k)}$, $b\in A^{(\ell)}$ where $k,\ell \in \lbrace 0,1\rbrace$. 
\end{example}
\begin{example}\label{ExamplePseudoInvolution}  Let $A=A^{(0)}\oplus A^{(1)}$ be a $\mathbb Z/2\mathbb Z$-graded algebra. A linear map $\star \colon A \to A$ is a \textit{pseudoinvolution} (see \cite{MartinezZelmanov}) if $\left(A^{(k)}\right)^\star = A^{(k)}$ for all $k=0,1$, $a^{\star\star}=(-1)^k a$, $(ab)^\star=(-1)^{k\ell}b^\star a^\star$ for all $a\in A^{(k)}$, $b\in A^{(\ell)}$ where $k,\ell \in \lbrace 0,1\rbrace$. 
\end{example}

Denote by $\widetilde{\Aut}(\Gamma)$ the set of all pseudoautomorphisms of $\Gamma$. 
For each $s,t \in T$ we have the following possibilities:
\begin{enumerate}
	\item\label{CaseGrPseudo1} $A^{(s)}A^{(t)}=A^{(t)}A^{(s)}=0$. In this case we may assume that for every $\varphi\in\widetilde{\Aut}(\Gamma)$ $$\alpha(s,t)=\beta(s,t)=\alpha(t,s)=\beta(t,s)=0.$$
	
	\item\label{CaseGrPseudo2} $A^{(s)}A^{(t)}\ne 0$, $A^{(t)}A^{(s)} = 0$. (If $A^{(s)}A^{(t)}= 0$, $A^{(t)}A^{(s)} \ne 0$, we just swap $s$ and $t$.) In this case  we may assume that $$\beta(s,t)=\alpha(t,s)=\beta(t,s)=0$$
	and $\alpha(s,t)\ne 0$ is defined uniquely. 
	
	\item\label{CaseGrPseudo3}  $s\ne t$, $A^{(s)}A^{(t)}\ne 0$, $A^{(t)}A^{(s)} \ne 0$ and for some $\lambda,\mu \in \mathbbm{k}$, where $\lambda\mu \ne 0$,  the graded polynomial identity
	\begin{equation}\label{EqGradedPseudoAutomorphismIdentity}\lambda\, x^{(s)}y^{(t)}+\mu\, y^{(t)}x^{(s)}\equiv 0\end{equation}
	holds in $A$. (If~\eqref{EqGradedPseudoAutomorphismIdentity} holds for some  $\lambda,\mu \in \mathbbm{k}$ where at least one of $\lambda$ and $\mu$ is nonzero, then both of $\lambda$ and $\mu$ are nonzero, since $A^{(s)}A^{(t)}\ne 0$, $A^{(t)}A^{(s)} \ne 0$.)  In this case $x^{(s)}y^{(t)}$ can be expressed in terms of $y^{(t)}x^{(s)}$ and vice versa. Hence we may assume that $$\beta(s,t)=\beta(t,s)=0$$
	and $\alpha(s,t)\ne 0$, $\alpha(t,s)\ne 0$ are defined uniquely.

	\item\label{CaseGrPseudo4} $s=t$, $A^{(s)}A^{(s)}\ne 0$ and for some $\lambda,\mu \in \mathbbm{k}$, where $\lambda\mu \ne 0$,  the graded polynomial identity
	\eqref{EqGradedPseudoAutomorphismIdentity} holds.   In this case $x^{(s)}y^{(s)}$ can be expressed in terms of $y^{(s)}x^{(s)}$ and vice versa. Hence we may assume that $\beta(s,s)=0$
	and $\alpha(s,s)\ne 0$ is defined uniquely.

\item\label{CaseGrPseudo5} $s\ne t$ and for any $\lambda,\mu \in \mathbbm{k}$, such that the graded polynomial identity
\eqref{EqGradedPseudoAutomorphismIdentity} holds, we have $\lambda=\mu=0$.  In this case the elements $\alpha(s,t)$, $\beta(s,t)$, $\alpha(t,s)$, $\beta(t,s)$ are defined uniquely.

	\item\label{CaseGrPseudo6} $s=t$ and for any $\lambda,\mu \in \mathbbm{k}$, such that the graded polynomial identity
	\eqref{EqGradedPseudoAutomorphismIdentity} holds, we have $\lambda=\mu=0$.  In this case the elements $\alpha(s,s)$ and $\beta(s,s)$ are defined uniquely.
\end{enumerate}	

\begin{proposition}\label{PropositionGrPseudoAutoDet2x2AlphaBetaNonzero} In the cases~\ref{CaseGrPseudo5} and~\ref{CaseGrPseudo6}
	we have \begin{equation}\label{EqDet2x2AlphaBetaNonzero}\left|\begin{smallmatrix}
	\alpha(s,t) & \beta(t,s) \\
	\beta(s,t) & \alpha(t,s)\end{smallmatrix} \right| \ne 0.\end{equation}
\end{proposition}
\begin{proof}
	The condition~\eqref{EqGradedPseudoAutomorphism}
	implies
	\begin{equation}\label{EqGradedPseudoAutomorphismTwo}
	\left\lbrace\begin{array}{l}
	\varphi(ab)=\alpha(s,t)\,\varphi(a)\varphi(b)+\beta(s,t)\,\varphi(b)\varphi(a),\\
	\varphi(ba)=\beta(t,s)\,\varphi(a)\varphi(b)+\alpha(t,s)\,\varphi(b)\varphi(a)\\
	\end{array}
	\right.
	\end{equation} for all $a\in A^{(s)}$ and $b\in A^{(t)}$.
	
	Suppose $\left|\begin{smallmatrix}
	\alpha(s,t) & \beta(t,s) \\
	\beta(s,t) & \alpha(t,s)\end{smallmatrix} \right| = 0$.
	Then the columns of the matrix $\left(\begin{smallmatrix}
	\alpha(s,t) & \beta(t,s) \\
	\beta(s,t) & \alpha(t,s)\end{smallmatrix} \right)$ are linearly dependent
	and there exist $\lambda, \mu \in \mathbbm{k}$, either $\lambda\ne 0$ or $\mu \ne 0$, such that
	$$\lambda \left(\begin{smallmatrix}
	\alpha(s,t)  \\
	\beta(s,t) \end{smallmatrix} \right) + \mu \left(\begin{smallmatrix}
	 \beta(t,s) \\
	\alpha(t,s)\end{smallmatrix} \right)= \left(\begin{smallmatrix}
	0 \\
	0 \end{smallmatrix} \right).$$

Adding the first equation of~\eqref{EqGradedPseudoAutomorphismTwo} multiplied by $\lambda$
to the second equation of~\eqref{EqGradedPseudoAutomorphismTwo} multiplied by $\mu$,
we get 
$$\lambda\,\varphi(ab)+\mu\,\varphi(ba)=0 \text{ for all }a\in A^{(s)}\text{ and }b\in A^{(t)}.$$
Taking into account that $\varphi$ is a linear bijection,
we obtain that the graded polynomial identity~\eqref{EqGradedPseudoAutomorphismIdentity}
holds in $A$. However this cannot happen in the cases~\ref{CaseGrPseudo5} and~\ref{CaseGrPseudo6}.
We get a contradiction. Therefore \eqref{EqDet2x2AlphaBetaNonzero} holds.
\end{proof}

For a vector space $V$ over a field $\mathbbm{k}$ denote by $\GL(V)$
the group of all $\mathbbm{k}$-linear invertible operators $V \to V$.

\begin{proposition}\label{PropositionGroupGradedPseudoAutomorphism} Let $\Gamma \colon A=\bigoplus_{t\in T} A^{(t)}$ be a grading on an algebra $A$ over a field $\mathbbm{k}$ by some set $T$.
		Then $\widetilde{\Aut}(\Gamma)$ is a subgroup of $\GL(A)$.\end{proposition}
\begin{proof}
 It is obvious that $\widetilde{\Aut}(\Gamma)$
is closed under the composition and contains the identity map. Hence it is sufficient to show that for every
$\varphi \in \widetilde{\Aut}(\Gamma)$ we have $\varphi^{-1} \in \widetilde{\Aut}(\Gamma)$.

We have to show that there exist such $\alpha'(s,t), \beta'(s,t), \alpha'(t,s), \beta'(t,s) \in \mathbbm{k}$
that
\begin{equation*}
\left\lbrace\begin{array}{l}
\varphi^{-1}(ab)=\alpha'(s,t)\,\varphi^{-1}(a)\varphi^{-1}(b)+\beta'(s,t)\,\varphi^{-1}(b)\varphi^{-1}(a),\\
\varphi^{-1}(ba)=\beta'(t,s)\,\varphi^{-1}(a)\varphi^{-1}(b)+\alpha'(t,s)\,\varphi^{-1}(b)\varphi^{-1}(a)\\
\end{array}
\right.
\end{equation*} for all $a\in A^{(s)}$ and $b\in A^{(t)}$.

We treat the cases listed above Proposition~\ref{PropositionGrPseudoAutoDet2x2AlphaBetaNonzero} separately.

In the case~\ref{CaseGrPseudo1} we can take $$\alpha'(s,t)= \beta'(s,t)= \alpha'(t,s)= \beta'(t,s)=0.$$

In the case~\ref{CaseGrPseudo2}  we can take $$\alpha'(s,t)=\frac{1}{\alpha(s,t)},\ \beta'(s,t)=0,\ \alpha'(t,s)=\beta'(t,s)=0.$$

 In the cases~\ref{CaseGrPseudo3}--\ref{CaseGrPseudo6} we can define $\alpha'(s,t), \beta'(s,t), \alpha'(t,s), \beta'(t,s) \in \mathbbm{k}$ by $$\left(\begin{smallmatrix}
\alpha'(s,t) & \beta'(s,t) \\
\beta'(t,s) & \alpha'(t,s)\end{smallmatrix}
\right) := \left(\begin{smallmatrix}
\alpha(s,t) & \beta(s,t) \\
\beta(t,s) & \alpha(t,s)\end{smallmatrix}
\right)^{-1}.$$ (Here we use Proposition~\ref{PropositionGrPseudoAutoDet2x2AlphaBetaNonzero}.)
\end{proof}

 Denote by $Q$ the subgroup
of $\GL_2(\mathbbm{k})$ that consists of all matrices $\left(\begin{smallmatrix}
\alpha & \beta \\
\beta & \alpha
\end{smallmatrix}\right)$ where $\alpha,\beta\in \mathbbm{k}$, $\alpha^2-\beta^2 \ne 0$.
If $\ch \mathbbm k \ne 2$, we have $Q \cong \mathbbm{k}^\times \times \mathbbm{k}^\times$ where $\left(\begin{smallmatrix}
\alpha & \beta \\
\beta & \alpha
\end{smallmatrix}\right) \mapsto (\alpha+\beta, \alpha-\beta)$.

Denote by $T_4$ and $T_6$ the subsets of $T$ that consists of the elements $s\in T$
that correspond to the cases~\ref{CaseGrPseudo4} and~\ref{CaseGrPseudo6}, respectively. Denote by $T_2$, $T_3$ and $T_5$
 the sets of the two-element subsets $\lbrace s,t \rbrace$ of $T$
 that correspond to the cases~\ref{CaseGrPseudo2}, \ref{CaseGrPseudo3} and~\ref{CaseGrPseudo5}, respectively.
 
 Now we define the following group homomorphisms corresponding to the elements of the sets $T_2, \ldots, T_6$.
 
  For every $\lbrace s,t\rbrace \in T_2$ we define
$\tau_{\lbrace s,t\rbrace} \colon \widetilde{\Aut}(\Gamma) \to \mathbbm{k}^\times$ by $\tau_{\lbrace s,t\rbrace}(\varphi):=\alpha(s,t)$. 
 
For every $\lbrace s,t\rbrace \in T_3$ we define
$\tau_{\lbrace s,t\rbrace} \colon \widetilde{\Aut}(\Gamma) \to \mathbbm{k}^\times \times \mathbbm{k}^\times$ by $\tau_{\lbrace s,t\rbrace}(\varphi):=(\alpha(s,t),\alpha(t,s))$. 
 
For every $s \in T_4$ we define
$\tau_s \colon \widetilde{\Aut}(\Gamma) \to \mathbbm{k}^\times$ by $\tau_s(\varphi):=\alpha(s,s)$. 
 
For every $\lbrace s,t\rbrace \in T_5$ we define
$\tau_{\lbrace s,t \rbrace} \colon \widetilde{\Aut}(\Gamma) \to \GL_2(\mathbbm{k})$ by $\tau_{\lbrace s,t\rbrace}(\varphi):=\left(\begin{smallmatrix}
\alpha(s,t) & \beta(t,s) \\
\beta(s,t) & \alpha(t,s)
\end{smallmatrix}\right)$. 
 
For every $s \in T_6$ we define
$\tau_s \colon \widetilde{\Aut}(\Gamma) \to Q$ by $\tau_s(\varphi):=\left(\begin{smallmatrix}
\alpha(s,s) & \beta(s,s) \\
\beta(s,s) & \alpha(s,s)
\end{smallmatrix}\right)$. 

Denote the group $$\prod\limits_{\omega \in T_2} \mathbbm{k}^\times \times
\prod\limits_{\omega \in T_3} \left(\mathbbm{k}^\times \times \mathbbm{k}^\times \right) \times
\prod\limits_{\omega \in T_4} \mathbbm{k}^\times \times
\prod\limits_{\omega \in T_5} \GL_2(\mathbbm{k}) \times
\prod\limits_{\omega \in T_6} Q$$
by $G(\Gamma)$
and define the homomorphism $$\tau \colon \widetilde{\Aut}(\Gamma) \to G(\Gamma)
$$
by  
$\pi_\omega\tau(\varphi)=
\tau_\omega(\varphi)$  for all $j=2,\ldots, 6$ and $\omega\in T_j$
 where  $\pi_\omega$ is the projection from $G(\Gamma)$ onto the $\omega$-component.
We get an exact sequence of groups
$$\xymatrix{
0 \ar[r] & \Aut(\Gamma) \ar[r] & \widetilde{\Aut}(\Gamma) \ar[r]^(0.5)\tau & G(\Gamma)}$$
where $\Aut(\Gamma)$ is the automorphism group of $\Gamma$.

Consider now the case when the algebra $A$ is finite dimensional and the base field $\mathbbm{k}$ is algebraically closed. In this case the sets $T_2, \ldots, T_6$ are finite. Denote by $H(\Gamma)$ the subset
of $\GL(A) \times G(\Gamma)$ that consists of all the elements $\left(\varphi, \bigl(\gamma_\omega \bigr)_{\substack{j=1,\ldots, 6, \\ \omega\in T_j}}\right)$ that satisfy \begin{itemize} \item the polynomial equations~\eqref{EqGradedPseudoAutomorphism}
 where $a$ and $b$ run the bases of $A^{(s)}$ and $A^{(t)}$, respectively;
 \item the conditions $\pi_\omega\tau(\varphi)=
\gamma_\omega$  for all $j=2,\ldots, 6$ and $\omega\in T_j$.
\end{itemize} Then $H(\Gamma)$ is an affine algebraic variety. Proposition~\ref{PropositionGroupGradedPseudoAutomorphism} implies that $H(\Gamma)$ is an affine algebraic group. Moreover, $H(\Gamma)$ maps bijectively on its image 
 $\widetilde{\Aut}(\Gamma)$ in the first component $\GL(A)$. By~\cite[Chapter II, \S 7.4, Proposition B(b)]{HumphreysAlgGr}, $\widetilde{\Aut}(\Gamma)$ is a closed subgroup of $\GL(A)$.
 
 An example of calculation of $\widetilde{\Aut}(\Gamma)$ will be given in Theorem~\ref{TheoremGrPseudoAutoM11k} below.
 
 It turns out that the existence of a graded pseudoautomorphism for given $\alpha$ and $\beta$ imposes certain restrictions on the corresponding graded algebra in terms of graded polynomial identities:
 
 \begin{theorem}\label{TheoremPseudoAutoGrIdentity}
 	Let $\Gamma \colon A = \bigoplus\limits_{g\in G} A^{(g)}$ be a grading on an associative
 	algebra $A$ over a field $\mathbbm{k}$ by a group $G$ and let $\varphi \in \widetilde{\Aut}(\Gamma)$
 	with $\alpha, \beta \colon \supp\Gamma \times \supp\Gamma \to \mathbbm{k}$.
 	Then the following graded polynomial identities hold in $A$: for every $g,h,t \in \supp \Gamma$ we have
 	\begin{equation}\label{EqPseudoAutoGrIdentity}\begin{split}
 	(\alpha(g,h)\alpha(gh,t)-\alpha(h,t)\alpha(g,ht))x^{(g)}y^{(h)}z^{(t)}-\beta(h,t)\alpha(g,ht)x^{(g)}z^{(t)}y^{(h)}
 	\\+ \beta(g,h)\alpha(gh,t)y^{(h)}x^{(g)}z^{(t)}-\alpha(h,t)\beta(g,ht)y^{(h)}z^{(t)}x^{(g)}\\+\alpha(g,h)\beta(gh,t)z^{(t)}x^{(g)}y^{(h)}+(\beta(g,h)\beta(gh,t)-\beta(h,t)\beta(g,ht))z^{(t)}y^{(h)}x^{(g)}\equiv 0.
 	\end{split}\end{equation}
 	In particular, if $\Gamma$ is trivial and $\alpha\ne 0$, $\beta \ne 0$, the ordinary polynomial identity $$[[x_1, x_2], x_3]\equiv 0$$ holds in $A$.	
 \end{theorem}
 \begin{proof} Let $a\in A^{(g)}$, $b\in A^{(h)}$, $c\in A^{(t)}$.
 	Then \begin{equation}\begin{split}\label{EqPseudoAutoGrIdentity1}
 	\varphi\bigl((ab)c\bigr) = \alpha(gh,t)\varphi(ab)\varphi(c)+\beta(gh,t)\varphi(c)\varphi(ab)\\
 	=  \alpha(g,h)\alpha(gh,t)\varphi(a)\varphi(b)\varphi(c)+
 	\beta(g,h)\alpha(gh,t)\varphi(b)\varphi(a)\varphi(c)\\
 	+\alpha(g,h)\beta(gh,t)\varphi(c)\varphi(a)\varphi(b)
 	+\beta(g,h)\beta(gh,t)\varphi(c)\varphi(b)\varphi(a)
 	\end{split}
 	\end{equation}
 	At the same time,
 	\begin{equation}\begin{split}\label{EqPseudoAutoGrIdentity2}
 	\varphi\bigl(a(bc)\bigr) = \alpha(g,ht)\varphi(a)\varphi(bc)+\beta(g,ht)\varphi(bc)\varphi(a) 
 	\\ = \alpha(h,t) \alpha(g,ht)\varphi(a)\varphi(b) \varphi(c) + \beta(h,t) \alpha(g,ht)\varphi(a) \varphi(c)\varphi(b)
 	\\ + \alpha(h,t) \beta(g,ht)\varphi(b)\varphi(c)\varphi(a) + \beta(h,t) \beta(g,ht)\varphi(c)\varphi(b)\varphi(a).
 	\end{split}
 	\end{equation}
 	Subtracting~\eqref{EqPseudoAutoGrIdentity2} from~\eqref{EqPseudoAutoGrIdentity1},
 	we get~\eqref{EqPseudoAutoGrIdentity}.
 	
 	When $\Gamma$ is trivial, $g=h=t=1$ and~\eqref{EqPseudoAutoGrIdentity} becomes
 	$$\alpha(1,1)\beta(1,1) [[z^{(1)}, x^{(1)}], y^{(1)}] \equiv 0,$$ which implies
 	that the ordinary polynomial identity $[[x_1, x_2], x_3]\equiv 0$ holds in $A$.	
 \end{proof}

 \section{Pseudoautomorphisms of trivially graded algebras}\label{SectionUngradedPseudoAuto}
 
 Note that if $\varphi \in \widetilde{\Aut}(\Gamma)$ for some group grading $\Gamma \colon A = \bigoplus\limits_{g\in G} A^{(g)}$, then $\varphi\bigr|_{A^{(1)}}$ is the pseudoautomorphism of the trivially graded algebra $A^{(1)}$ where $1$ is the identity element in~$G$. Therefore, in order to understand the case of nontrivial gradings,
 it is useful to consider first pseudoautomorphisms of trivially graded algebras. 
 In this section we treat the case of a trivial grading in detail. 
  
 Let $A$ be an algebra over a field $\mathbbm{k}$.
 We say that an operator $\varphi \in \GL(A)$
 is a \textit{pseudoautomorphism} of $A$ if there exist $\alpha, \beta \in \mathbbm{k}$
 such that \begin{equation}\label{EqVarphiAB}\varphi(ab)=\alpha\, \varphi(a)\varphi(b)+\beta \,\varphi(b)\varphi(a)
 \text{ for all }a,b\in A.\end{equation}
 Denote the group of all pseudoautomorphisms of $A$ by $\widetilde{\Aut}(A)$.
 
 Suppose now that $A$ is associative and the identity $yx\equiv \gamma\, xy$ holds in $A$ for some $\gamma\in \mathbbm{k}$.
 Then \begin{equation*}\begin{split}\gamma\, x(yz) \equiv (yz)x \equiv y(zx) \equiv \gamma\, y(xz) \equiv \gamma\, (yx)z  \equiv
 		\gamma^2 (xy)z.\end{split}\end{equation*}
 Hence either $A^3=0$ or $\gamma = \gamma^2$, i.e. $\gamma = 1$, and $A$ is commutative.
 In the case $A^3=0$ but $A^2\ne 0$ we get $yx\equiv \gamma xy \equiv \gamma^2 yx$,
 $\gamma = \pm 1$, and $A$ is either commutative or anti-commutative.
 
 \begin{remark}\label{RemarkCommAnticomm}
 	If $A$ is either commutative or anti-commutative, we can always assume that
 	$\beta = 0$. In this case $\alpha \varphi$ is an automorphism of $A$.
 	Therefore, for an (anti-)commutative algebra~$A$ with a nonzero multiplication the set $\widetilde{\Aut}(A)$ is a subgroup of $\GL(A)$
 	isomorphic to $\Aut(A)\times \mathbbm{k}^\times$.
 \end{remark}
 
 Note that if $\alpha = \pm \beta$ for some $\varphi \in \widetilde{\Aut}(A)$, then $\varphi(ab)= \pm \varphi(ba)$
 for all $a,b \in A$, i.e. the algebra $A$ is either commutative or anti-commutative. 
 
 If the algebra $A$ is neither commutative, nor anti-commutative, 
 the elements $\alpha$ and $\beta$ are uniquely determined for each $\varphi \in \widetilde{\Aut}(A)$. In this case~\eqref{EqVarphiAB}
 applied for both $\varphi(ab)$ and $\varphi(ba)$
 yields \begin{equation}\label{EqVarphiABInverse}\varphi(a)\varphi(b)=\frac{\alpha}{\alpha^2-\beta^2} \varphi(ab)- \frac{\beta}{\alpha^2-\beta^2} \varphi(ba)
 \text{ for all }a,b\in A.\end{equation}
 
 Assign to each $\varphi \in \widetilde{\Aut}(A)$ the matrix $\tau(\varphi):=\left(\begin{smallmatrix}
 \alpha & \beta \\
 \beta  & \alpha
 \end{smallmatrix}\right)$. Again we get an exact sequence of groups
 
 $$0 \to \Aut(A) \to \widetilde{\Aut}(A) \xrightarrow{\tau} \left\lbrace
 \left(\begin{smallmatrix}
 \alpha & \beta \\
 \beta  & \alpha
 \end{smallmatrix}\right)
 \mathrel{\bigl|} \alpha^2-\beta^2 \ne 0 \right\rbrace$$
 
 Below we show that for some algebras $A$ we have  $$\im\tau = \left\lbrace
 \left(\begin{smallmatrix}
 \alpha & \beta \\
 \beta  & \alpha
 \end{smallmatrix}\right)
 \mathrel{\bigl|} \alpha^2-\beta^2 \ne 0 \right\rbrace,$$ while for other algebras
 we have $$\im\tau = \left\lbrace
 \left(\begin{smallmatrix}
 \alpha & \beta \\
 \beta  & \alpha
 \end{smallmatrix}\right)
 \mathrel{\bigl|} \alpha^2-\beta^2 \ne 0, \alpha\beta=0 \right\rbrace,$$
 i.e. all pseudoautomorphisms are just scalar multiples of automorphisms and anti-automorphisms.
 
 \begin{proposition}\label{PropositionUnitPseudoAuto} Let $A$ be a unital non-commutative algebra, $\varphi \in  \widetilde{\Aut}(A)$
 	and $\alpha, \beta \in \mathbbm{k}$ such that~\eqref{EqVarphiAB} holds.
 	Then $\varphi(1_A)=\frac{1_A}{\alpha+\beta}$.
 \end{proposition}
 \begin{proof}
 	Note that for all $a\in A$
 	we have
 	\begin{equation}\label{EqVarphi1a1}\varphi(a)=\varphi(1_A\cdot a)=\alpha \varphi(1_A)\varphi(a)+\beta \varphi(a)\varphi(1_A)
 	\end{equation}
 	and
 	\begin{equation}\label{EqVarphi1a2}\varphi(a)=\varphi(a\cdot 1_A)=\alpha \varphi(a)\varphi(1_A)+\beta \varphi(1_A)\varphi(a).\end{equation}
 	Subtracting~\eqref{EqVarphi1a2} from~\eqref{EqVarphi1a1}, we get
 	$$(\alpha-\beta)(\varphi(1_A)\varphi(a)-\varphi(a)\varphi(1_A))=0.$$
 	As we have already mentioned above, the condition $\alpha = \beta$
 	implies the commutativity of $A$. Hence we may assume that $\alpha \ne \beta$
 	and $\varphi(1_A)\varphi(a)=\varphi(a)\varphi(1_A)$
 	for all $a\in A$. In particular,
 	\eqref{EqVarphi1a1} implies
 	\begin{equation*}\varphi(a)=(\alpha+\beta) \varphi(1_A)\varphi(a)= (\alpha+\beta) \varphi(a)\varphi(1_A).
 	\end{equation*}
 	Since $\varphi$ is invertible, $1_A=\varphi(a)$ for some $a\in A$.
 	Therefore, we get $1_A=(\alpha+\beta)\varphi(1_A)$,
 	and $\varphi(1_A)=\frac{1_A}{\alpha+\beta} $.
 \end{proof}
 
  It turns out that for a wide class of algebras pseudoautomorphisms are just
 scalar multiples of automorphisms and anti-automorphisms:

 \begin{theorem}\label{Theorem2x2PseudoTrivial}
 	If an algebra $A$ over a field $\mathbbm{k}$ contains a subalgebra isomorphic to 
 	the algebra $\UT_2(\mathbbm{k})$ of upper triangular $2\times 2$ matrices,
 	then all its ungraded pseudoautomorphisms
 	are just scalar multiples of their automorphisms and anti-automorphisms, i.e. $$\widetilde{\Aut}(A)\cong \Aut^*(A)\times \mathbbm{k}^\times$$ where $\Aut^*(A)$ is the group of automorphisms and anti-automorphisms of $A$.
 \end{theorem}
 \begin{proof}
 	It is sufficient to notice that for the matrix units $e_{11}, e_{12}, e_{22}$
 	we have $[[e_{11}, e_{12}], e_{22}] = e_{12}\ne 0$
 	and apply Theorem~\ref{TheoremPseudoAutoGrIdentity}.
 \end{proof}

  Below we denote by $N \leftthreetimes H$ the semi-direct product of groups $N$ and $H$ where $N$ is normal. 
\begin{corollary}
 	All ungraded pseudoautomorphisms of the full matrix algebra $M_n(\mathbbm{k})$
 	as well as the algebra $\UT_n(\mathbbm{k})$ of upper triangular $n\times n$ matrices
 	are just scalar multiples of their automorphisms and anti-automorphisms. Moreover, for $n\geqslant 2$
 	$$\widetilde{\Aut}(M_n(\mathbbm{k}))\cong \Aut^*(M_n(\mathbbm{k}))\times \mathbbm{k}^\times
 	\cong  \bigl(\mathrm{PGL}_n(\mathbbm{k}) \leftthreetimes C_2\bigr) \times \mathbbm{k}^\times$$
 	where  the cyclic group $C_2$ of order $2$ is generated by the operator of transposition $(-)^T$.
 \end{corollary}
 \begin{proof}
 	The case $n=1$ is trivial.
 	The case $n\geqslant 2$ follows from Theorem~\ref{Theorem2x2PseudoTrivial}.
 \end{proof}
 
 Recall~\cite{KrakReg} that $[[x_1, x_2], x_3]\equiv 0$ generates the polynomial identities for the infinitely generated Grassmann (or exterior) algebra over a field $\mathbbm{k}$, $\ch \mathbbm{k} = 0$. Now we give a description for the group of its ungraded pseudoautomorphisms:
 
 \begin{theorem}\label{TheoremPseudoGrassmann} Let $E$ be the ungraded unital Grassmann algebra over a field $\mathbbm{k}$ with generators $e_k$ where $k\in \mathbb N$
 	and $\ch \mathbbm{k} \ne 2$.
 	Then $$\widetilde{\Aut}(E)\cong \Aut(E) \leftthreetimes (\mathbbm{k}^\times \times \mathbbm{k}^\times)$$
 	where the group $\mathbbm{k}^\times \times \mathbbm{k}^\times \cong Q = \left\lbrace \left( \left. \begin{smallmatrix} \alpha & \beta \\ \beta & \alpha \end{smallmatrix}\right) \right| \alpha, \beta\in \mathbbm{k},\ \alpha^2-\beta^2\ne 0 \right\rbrace$
 	is acting on $E$ in the following way: 
 	$$ \left(  \begin{smallmatrix} \alpha & \beta \\ \beta & \alpha \end{smallmatrix}\right)(e_{i_1} \ldots e_{i_k})= \left(\frac{\alpha-\beta}{\alpha+\beta}\right)^{\left[\frac{k}{2}\right]}(\alpha+\beta)^{k-1}
 	\,e_{i_1} \ldots e_{i_k}$$
 	and $\tau\left(  \begin{smallmatrix} \alpha & \beta \\ \beta & \alpha \end{smallmatrix}\right) = \left(  \begin{smallmatrix} \alpha & \beta \\ \beta & \alpha \end{smallmatrix}\right)$. 
 \end{theorem}
\begin{proof}
	An explicit verification shows that $ \left(  \begin{smallmatrix} \alpha & \beta \\ \beta & \alpha \end{smallmatrix}\right)$ is indeed acting on $E$ by a pseudoautomorphism with given $\alpha$ and $\beta$. Hence
	we get a split exact sequence
	$$0 \to \Aut(A) \to \widetilde{\Aut}(A) \xrightarrow{\tau} Q \to 0$$
\end{proof}	

Another example of an algebra $A$ with $\tau\left(\widetilde{\Aut}(A)\right)=\left\lbrace
\left(\begin{smallmatrix}
\alpha & \beta \\
\beta  & \alpha
\end{smallmatrix}\right)
\mathrel{\bigl|} \alpha^2-\beta^2 \ne 0 \right\rbrace$ will be given in Theorem~\ref{TheoremPseudoAutoImageTauIsQNilpAlg} below.

 \section{Graded pseudoautomorphisms of $M_{1,1}(\mathbbm{k})$}
 
 By $\langle a \rangle_n$ we denote the cyclic group generated by an element $a$ of order $n$.
 
\begin{theorem}\label{TheoremGrPseudoAutoM11k}
	Let $\mathbbm{k}$ be a field and let $M_{1,1}(\mathbbm{k}) = M_{1,1}(\mathbbm{k})^{(0)}\oplus M_{1,1}(\mathbbm{k})^{(1)}$
	be the full $2\times 2$ matrix algebra with the $\mathbb Z/2 \mathbb Z$ grading defined by 
	$$M_{1,1}(\mathbbm{k})^{(0)} = \left\lbrace \left( \left. \begin{smallmatrix} \gamma_{11} & 0 \\ 0 & \gamma_{22} \end{smallmatrix}\right) \right| \gamma_{11}, \gamma_{22}\in \mathbbm{k} \right\rbrace,\quad
	M_{1,1}(\mathbbm{k})^{(1)} = \left\lbrace \left( \left. \begin{smallmatrix} 0 & \gamma_{12} \\ \gamma_{21} & 0 \end{smallmatrix}\right) \right| \gamma_{12}, \gamma_{21}\in \mathbbm{k} \right\rbrace.$$
	Then \begin{equation}\begin{split}
	\widetilde\Aut\bigl(M_{1,1}(\mathbbm{k})\bigr) \cong
	\mathbbm{k}^\times \times \Bigl(\bigl(\mathbbm{k}^\times \times \mathbbm{k}^\times\bigr)  \leftthreetimes \left\langle (-)^T \right\rangle_2 \Bigr) \times \langle c \rangle_2 \\
	= \Bigl\lbrace (\lambda, \mu, \nu, a, b) \mathrel{\Bigl|} \lambda, \mu, \nu \in \mathbbm{k}^\times,\ a\in \left\langle (-)^T \right\rangle_2,\ b\in \langle c \rangle_2 \Bigr\rbrace   
	\end{split}
	\end{equation} where 
	$$(\lambda, \mu, \nu, 1, 1)\left( \begin{smallmatrix} \gamma_{11} & \gamma_{12} \\ \gamma_{21} & \gamma_{22} \end{smallmatrix}\right) := \left( \begin{smallmatrix} \lambda\, \gamma_{11} & \mu\, \gamma_{12} \\ \nu\, \gamma_{21} & \lambda\, \gamma_{22} \end{smallmatrix}\right),$$
	$$\Bigl(1, 1, 1,  (-)^T, 1\Bigr)\left( \begin{smallmatrix} \gamma_{11} & \gamma_{12} \\ \gamma_{21} & \gamma_{22} \end{smallmatrix}\right) :=\left( \begin{smallmatrix} \gamma_{11} & \gamma_{12} \\ \gamma_{21} & \gamma_{22} \end{smallmatrix}\right)^T=\left( \begin{smallmatrix} \gamma_{11} & \gamma_{21} \\ \gamma_{12} & \gamma_{22} \end{smallmatrix}\right),$$
	$$(1, 1, 1, 1, c)\left( \begin{smallmatrix} \gamma_{11} & \gamma_{12} \\ \gamma_{21} & \gamma_{22} \end{smallmatrix}\right) :=\left( \begin{smallmatrix} \gamma_{22} & \gamma_{12} \\ \gamma_{21} & \gamma_{11} \end{smallmatrix}\right)$$
	and $(-)^T$ is acting on $\mathbbm{k}^\times \times \mathbbm{k}^\times$ by swapping the components.
	In particular,  $\widetilde\Aut\bigl(M_{1,1}(\mathbbm{k})\bigr)$
	is generated by $\Aut^*\bigl(M_{1,1}(\mathbbm{k})\bigr)$ and the group $\Bigl\lbrace (\lambda, \mu, \nu, 1, 1) \mathrel{\Bigl|} \lambda, \mu, \nu \in \mathbbm{k}^\times \Bigr\rbrace$ of all operators that are scalar on homogeneous components. (By $\widetilde\Aut\bigl(M_{1,1}(\mathbbm{k})\bigr)$ and $\Aut^*\bigl(M_{1,1}(\mathbbm{k})\bigr)$
	we denote the groups of graded pseudoautomorphisms and of graded automorphisms and anti-automorphisms, respectively.)
\end{theorem}
\begin{proof} It is easy to see that the subgroup $ \lbrace (\lambda, \mu, \nu, a, b) \rbrace$ indeed exists
	in $\widetilde\Aut\bigl(M_{1,1}(\mathbbm{k})\bigr)$. Therefore it is sufficient to show that an arbitrary 
	$\varphi \in \widetilde\Aut\bigl(M_{1,1}(\mathbbm{k})\bigr)$ belongs to this subgroup.

	Let $\varphi \in \widetilde\Aut\bigl(M_{1,1}(\mathbbm{k})\bigr)$. 	
	Note that $ M_{1,1}(\mathbbm{k})^{(0)} \cong \mathbbm{k} \times \mathbbm{k}$ is commutative. Hence, by Remark~\ref{RemarkCommAnticomm}, up to scalar multiples the restriction of $\varphi$ on $M_{1,1}(\mathbbm{k})^{(0)}$ is an automorphism, however 
	$ \mathbbm{k} \times \mathbbm{k}$  admits only two automorphisms: the identical one
	and $c$. If $\varphi$ is acting by~$c$, then multiplying $\varphi$ by $c$, if necessary, we may assume that $\varphi\bigr|_{M_{1,1}(\mathbbm{k})^{(0)}} =\id_{M_{1,1}(\mathbbm{k})^{(0)}}$.
	
	Denote by $\varphi_{ij}^{k\ell} \in \mathbbm{k}$
	such elements that 	
	$$\varphi(e_{12})=\varphi^{12}_{12} e_{12} + \varphi^{21}_{12} e_{21},\qquad \varphi(e_{21})=\varphi^{12}_{21} e_{12} + \varphi^{21}_{21} e_{21}.$$
	
	By~\eqref{EqGradedPseudoAutomorphism}, 
	\begin{equation*}\begin{split}\varphi(e_{12})=\varphi((e_{11}+e_{22}) e_{12}) = \alpha(0,1)\, (e_{11}+e_{22}) \varphi(e_{12}) + \beta(0,1)\,  \varphi(e_{12})(e_{11}+e_{22})\\= (\alpha(0,1) + \beta(0,1))\varphi(e_{12}),\end{split}
	\end{equation*}
	\begin{equation*}\begin{split}\varphi(e_{12})=\varphi(e_{11} e_{12}) = \alpha(0,1)\, e_{11} \varphi(e_{12}) + \beta(0,1)\,  \varphi(e_{12})e_{11}= \alpha(0,1)\varphi^{12}_{12} e_{12} + \beta(0,1)\varphi^{21}_{12} e_{21}.\end{split}
	\end{equation*}
	Since $\varphi$ is a non-degenerate linear operator, we get $ \alpha(0,1)+\beta(0,1)=1$ and either $ \alpha(0,1)=1, \beta(0,1)=0$ or $ \alpha(0,1)=1, \beta(0,1)=0$.
	Therefore, either $\varphi(e_{12}) \in \mathbbm{k} e_{12}$ or $\varphi(e_{12}) \in \mathbbm{k} e_{21}$.
	Analogously,  either $\varphi(e_{21}) \in \mathbbm{k} e_{21}$ or $\varphi(e_{21}) \in \mathbbm{k} e_{12}$.
	Multiplying $\varphi$ by $(-)^T$, if necessary, we may assume that  $\varphi \in \lbrace (1, \mu, \nu, 1, 1) \mid \mu, \nu\in  \mathbbm{k}^\times \rbrace$ and the theorem is proven.
\end{proof}	
\begin{remark}
	\begin{equation*}\begin{split}
	\Aut^*\bigl(M_{1,1}(\mathbbm{k})\bigr) \cap \Bigl\lbrace (\lambda, \mu, \nu, 1, 1) \mathrel{\Bigl|} \lambda, \mu, \nu \in \mathbbm{k}^\times \Bigr\rbrace \\ = \Aut\bigl(M_{1,1}(\mathbbm{k})\bigr) \cap \Bigl\lbrace (\lambda, \mu, \nu, 1, 1) \mathrel{\Bigl|} \lambda, \mu, \nu \in \mathbbm{k}^\times \Bigr\rbrace=
	\Bigl\lbrace (1, \mu, \mu^{-1}, 1, 1) \mathrel{\Bigl|} \mu \in \mathbbm{k}^\times \Bigr\rbrace.
	\end{split}\end{equation*}
\end{remark}

 \section{Algebras with a generalized $H$-action}\label{SectionHmodGen}
 
  Let $H$ be an arbitrary associative algebra with $1$ over a field $\mathbbm{k}$.
We say that a (not necessarily associative) algebra $A$ is an algebra with a \textit{generalized $H$-action}
if $A$ is a left $H$-module
and for every $h \in H$ there exist some $k\in \mathbb N$ and some $h'_i, h''_i, h'''_i, h''''_i \in H$, $1\leqslant i \leqslant k$,
such that 
\begin{equation}\label{EqGeneralizedHopf}
h(ab)=\sum_{i=1}^k\bigl((h'_i a)(h''_i b) + (h'''_i b)(h''''_i a)\bigr) \text{ for all } a,b \in A.
\end{equation}

Equivalently, there exist linear maps $\Delta, \Theta \colon H \to H\otimes H$ (not necessarily coassociative)
such that 
$$ h(ab)=\sum\bigl((h_{(1)} a)(h_{(2)} b) + (h_{[1]} b)(h_{[2]} a)\bigr) \text{ for all } a,b \in A.$$ (Here we use the notation $\Delta(h)= \sum h_{(1)} \otimes h_{(2)}$ and $\Theta(h)= \sum  h_{[1]} \otimes h_{[2]}$.)

\begin{example}\label{ExampleHmodule} 
An algebra $A$
over a field $\mathbbm{k}$
is a \textit{(left) $H$-module algebra}
for some Hopf algebra~$H$
if $A$ is endowed with a structure of a (left) $H$-module such that
$h(ab)=(h_{(1)}a)(h_{(2)}b)$
for all $h \in H$, $a,b \in A$. Here we use Sweedler's notation
$\Delta h = \sum h_{(1)} \otimes h_{(2)}$ where $\Delta$ is the comultiplication
in $H$ and the symbol of the sum is omitted. If $A$ is an $H$-module algebra,
then $A$ is an algebra with a generalized $H$-action.
\end{example}

\begin{example}\label{ExampleIdFT}
Recall that if $T$ is a semigroup, then the \textit{semigroup algebra} $\mathbbm{k}T$ over a field~$\mathbbm{k}$ is the vector space with the formal basis $(t)_{t\in T}$ and the multiplication induced by the one in $T$.
Let $A$ be an associative algebra with an action 
of a monoid $T$ by endomorphisms and anti-endomorphisms. Then $A$ is an algebra with
 a generalized $\mathbbm{k}T$-action.
\end{example}

\begin{example}\label{ExampleGr}
Let $\Gamma \colon A=\bigoplus_{t\in T} A^{(t)}$ be a grading on an algebra for some set $T$. Then $\Gamma$ defines on $T$
a partial operation $\star$ with the domain $T_0:=\lbrace (s,t) \mid A^{(s)}A^{(t)} \ne 0 \rbrace$ by $s\star t := r$
where $A^{(s)}A^{(t)} \subseteq A^{(r)}$.
Consider the algebra $\mathbbm{k}^T$ of all functions from $T$ to $\mathbbm{k}$ with pointwise operations.
Then $\mathbbm{k}^T$ acts on $A$ naturally: $ha = h(t)a$ for all $a\in A^{(t)}$.
Denote by $\zeta \colon \mathbbm{k}^T \to \End_\mathbbm{k}(A)$
the corresponding homomorphism.
Let $h_t(s):=\left\lbrace\begin{smallmatrix} 1 & \text{if} & s=t,\\ 0 & \text{if} & s\ne t.\end{smallmatrix} \right.$
If $\supp \Gamma$ is finite, $T_0$ is finite too and we have
\begin{equation}\label{EqIdentityHFiniteSupp}h_r(ab)=\sum_{\substack{(s,t)\in T_0,\\ r=s\star t}}
(h_s a)(h_t b). \end{equation}
(Since the expressions in both sides are linear in $a$ and $b$, it is sufficient to check~\eqref{EqIdentityHFiniteSupp} only for homogeneous $a,b$.) In this case $(h_t)_{t\in \supp \Gamma}$
generates $\mathbbm{k}^T$ modulo $\ker \zeta$ as a $\mathbbm{k}$-vector space and, by the linearity, \eqref{EqIdentityHFiniteSupp} implies~(\ref{EqGeneralizedHopf}) for every $h\in \mathbbm{k}^T$. Therefore, in the case of a finite support of the grading, $A$ is an algebra with a generalized $\mathbbm{k}^T$-action.
\end{example}

\begin{remark}
	If $T$ is a group and $A$ is a $T$-graded algebra, then $A$ is a $\mathbbm{k} T$-comodule algebra
	where $\mathbbm{k} T$ is the group Hopf algebra. If $T$ is finite, then $A$ is a $(\mathbbm{k} T)^*$-module algebra
	where the $(\mathbbm{k} T)^*$-action corresponds to the generalized $\mathbbm{k}^T$-action defined above under
	the natural algebra isomorphism $(\mathbbm{k} T)^* \cong \mathbbm{k}^T$.
\end{remark}

 \section{Generalized $H$-actions compatible with gradings}\label{SectionGradedActions}

Let $\Gamma \colon A=\bigoplus_{t\in T} A^{(t)}$ be a grading on some algebra $A$ over a field $\mathbbm{k}$ by a set $T$. Suppose that $A$ is endowed with a generalized $H$-action for some unital associative algebra $H$. We say that this $H$-action is \textit{compatible} with the $T$-grading $\Gamma$
if this $H$-action preserves the components of $\Gamma$,
i.e. if $H A^{(t)} \subseteq A^{(t)}$ for all $t\in T$.

\begin{example}\label{ExampleSuperInvolution2} 
Every algebra with a superinvolution (see Example~\ref{ExampleSuperInvolution}) is naturally endowed
with an action of the group algebra of the cyclic group of order $2$ compatible
with the $\mathbb Z/2\mathbb Z$-grading.
\end{example}
\begin{example}\label{ExamplePseudoInvolution2} 
Every algebra with a pseudoinvolution (see Example~\ref{ExamplePseudoInvolution}) is naturally endowed
with an action of the group algebra of the cyclic group of order $4$ compatible
with the $\mathbb Z/2\mathbb Z$-grading.
\end{example}

As it was noticed by R.\,B. dos Santos in~\cite{dosSantos}, a superinvolution is reduced to a generalized $H$-action.

The same observation can be made in general. Recall that $\mathbbm{k}^T$ is the algebra of all functions $T\to \mathbbm{k}$ with pointwise operations.
\begin{theorem}\label{TheoremGradGenActionReplace} 
Let $\Gamma \colon A=\bigoplus_{t\in T} A^{(t)}$ be a grading on some algebra $A$ over a field $\mathbbm{k}$ by a set $T$. Suppose that $\supp \Gamma$ is finite and $A$ is endowed with a generalized $H$-action for some unital associative algebra $H$
such that $H A^{(t)} \subseteq A^{(t)}$ for all $t\in T$.
Then the formula $$(q\otimes h)a:=q(t)(ha)\text{ where }a\in A^{(t)},\ t\in T,\ q\in \mathbbm{k}^T,\ h\in H$$
defines on $A$ a generalized $\mathbbm{k}^T \otimes H$-action such that
$\mathbbm{k}^T \otimes H$-submodules  of $A$ are exactly 
 $T$-graded $H$-submodules.
\end{theorem}
\begin{proof}
Analogously to Example~\ref{ExampleGr} define $q_t \in \mathbbm{k}^T$ for all $t\in T$
by $q_t(s):=\left\lbrace\begin{smallmatrix} 1 & \text{for} & s=t,\\ 0 & \text{for} & s\ne t\end{smallmatrix} \right.$ and define on $T$
a partial operation $\star$ with the domain $T_0:=\lbrace (s,t) \mid A^{(s)}A^{(t)} \ne 0 \rbrace$
by $s\star t := r$
where $A^{(s)}A^{(t)}\subseteq A^{(r)}$.
Now~\eqref{EqGeneralizedHopf} applied to the generalized $H$-action
on $A$ implies that for every $h \in H$ there exists $k\in \mathbb N$ and $h'_i, h''_i, h'''_i, h''''_i \in H$, where $1\leqslant i \leqslant k$,
such that
\begin{equation}\label{EqGradedActionGenHAction}
(q_r \otimes h)(ab)=\sum_{i=1}^k \sum_{\substack{(s,t)\in T_0,\\ r=s\star t}} \Bigl(\bigl((q_s\otimes h'_i) a\bigr)\bigl((q_t\otimes h''_i) b\bigr) + \bigl((q_t\otimes h'''_i) b\bigr)\bigl((q_s\otimes h''''_i) a\bigr)\Bigr)
\end{equation}
 for all $r\in T$ and $a,b\in A$.
(Since the both sides of~\eqref{EqGradedActionGenHAction}
are linear in $a$ and $b$, it is sufficient to check this equality for homogeneous $a$ and $b$.)
Note that modulo the kernel of the action $\mathbbm{k}^T \otimes H
\to \End_\mathbbm{k}(A)$ the algebra $\mathbbm{k}^T \otimes H$
is the linear span of the elements $q_r \otimes h$ where $r\in \supp \Gamma$ and $h\in H$.
Now~\eqref{EqGeneralizedHopf} for the $\mathbbm{k}^T \otimes H$-action follows from~\eqref{EqGradedActionGenHAction}.
\end{proof}

It is obvious that in the case when the support of the grading consists of a single element,
such a generalized $\mathbbm{k}^T \otimes H$-action
reduces to just a generalized $H$-action for the same algebra~$H$.
Therefore examples of associative algebras with generalized $H$-actions
having non-$H$-invariant Jacobson radicals (see Examples 1--3 in~\cite{ASGordienko13} combined with Example~\ref{ExampleGr} from the current article)
show that in algebras with a generalized $H$-action
compatible with a $T$-grading the Jacobson radical is not necessarily
a $\mathbbm{k}^T \otimes H$-module.

Now we would like to find sufficient conditions on an $H$-action that 
force the Jacobson radical to be $\mathbbm{k}^T \otimes H$-invariant.
Inspired by Examples~\ref{ExampleSuperInvolution}
and~\ref{ExamplePseudoInvolution}, consider the following rather general case.

\begin{definition}\label{DefGradedAction}
Let $G$ be a group and let $\Gamma \colon A=\bigoplus_{t\in T} A^{(t)}$ be a grading on an algebra $A$ over a field $\mathbbm{k}$
by a group $T$. A group homomorphism
$G \to \GL(A)$ is called a \textit{$T$-graded $G$-action}
if there exist functions 
$\alpha, \beta \colon G \times T \times T \to \mathbbm{k}$  such that $gA^{(t)}\subseteq A^{(t)}$ and
\begin{equation}\label{EqGradedAction}
g(ab)=\alpha(g,s,t)(ga)(gb)+\beta(g,s,t)(gb)(ga)
\end{equation}
  for all $g\in G$, $s,t\in T$, $a\in A^{(s)}$ and $b\in A^{(t)}$.
  \end{definition}
  
  \begin{remark}
  Condition~\eqref{EqGradedAction} implies that $\beta(g,s,t)=0$ for all $g\in G$ and $s,t\in T$
such that $A^{(t)}A^{(s)}\ne 0$ and $st\ne ts$.
  \end{remark}
\begin{remark}
Every $T$-graded $G$-action is just a homomorphism $G\to \widetilde{\Aut}(\Gamma) $.
  \end{remark}
  
  Theorem~\ref{TheoremGradGenActionReplace}
implies that every $T$-graded $G$-action
is reduced to a generalized $\mathbbm{k}^T\otimes \mathbbm{k}G$-action.
Below we prove that the Jacobson radical of a finite dimensional associative
algebra is invariant under any such action.

We first show that the image of every graded two-sided ideal is again a graded two-sided ideal
and if the ideal is nilpotent, its image is nilpotent too.

\begin{lemma}\label{LemmaGradedActionIdeal} Let $A=\bigoplus_{t\in T} A^{(t)}$
be an algebra over a field $\mathbbm{k}$ with a $T$-graded $G$-action 
for some groups $G$ and $T$.
Then for every $g\in G$ and a graded two-sided ideal $I\subseteq A$
the space $gI$ 
is again a graded two-sided ideal. Moreover, if $I$ is nilpotent, then $gI$ is nilpotent too.
\end{lemma}
\begin{proof} Every $a \in I$
is a sum of its homogeneous components, which belong to $I$, since $I$ is graded.
Applying to $a$ an arbitrary element $g\in G$ and taking into account that the $G$-action
preserves homogeneous components, we obtain that
$ga$ is a sum of  homogeneous elements that belong to $gI$. Hence the space $gI$ is graded.
 
Now, in order to prove that $gI$ is a two-sided ideal, it is sufficient
to show that $ab, ba \in gI$ for all $a \in A^{(s)}$, $b\in (gI)\cap A^{(t)}$ where $s,t \in T$.
By~\eqref{EqGradedAction} we have
$$g^{-1}(ab)= \alpha(g^{-1},s,t)(g^{-1}a)(g^{-1}b)+\beta(g^{-1},s,t)(g^{-1}b)(g^{-1}a) \in I$$
and
$$g^{-1}(ba)= \alpha(g^{-1},t,s)(g^{-1}b)(g^{-1}a)+\beta(g^{-1},t,s)(g^{-1}a)(g^{-1}b) \in I.$$
Hence $gI$ is a graded two-sided ideal.

Suppose that a product of arbitrary  $n$  elements of $I$ with an arbitrary arrangement
of brackets equals $0$. We claim that $gI$ has the same property.
 Indeed, let $a_1, \ldots, a_n \in gI$ be homogeneous elements. Consider $g^{-1}(a_1 \ldots a_n)$
with an arbitrary arrangement of brackets on
$a_1 \ldots a_n$.
Applying~\eqref{EqGradedAction} $(n-1)$ times,
we obtain that  $g^{-1}(a_1 \ldots a_n)$
is a linear combination of products of $g^{-1}a_i$ taken in an arbitrary order.
Since $g^{-1}a_i \in I$ for all $i=1,\ldots, n$ and each product consists of $n$
multipliers, all such products equal zero, and $g^{-1}(a_1 \ldots a_n)=0$.
Hence $a_1 \ldots a_n=0$ and the ideal $gI$ is nilpotent too.
\end{proof}

Using Lemma~\ref{LemmaGradedActionIdeal}, we now deduce the invariance of the Jacobson radical: 

\begin{theorem}\label{TheoremTGradedGActionInvRad} Let $A=\bigoplus_{t\in T} A^{(t)}$
be a finite dimensional associative algebra over a field $\mathbbm{k}$ with a $T$-graded $G$-action 
for some groups $G$ and $T$. Suppose that either $\ch \mathbbm{k} = 0$ or $\ch \mathbbm{k} > \dim_\mathbbm{k} A$.
Then $$\bigl(\mathbbm{k}^T\otimes \mathbbm{k}G\bigr)J(A)\subseteq J(A).$$
\end{theorem}
\begin{proof}
The Jacobson radical $J(A)$ is a graded (see, e.g., Corollary~3.3 in~\cite{ASGordienko9})
and, therefore, a $\mathbbm{k}^T$-invariant ideal. By Lemma~\ref{LemmaGradedActionIdeal} the radical $J(A)$
is $G$-invariant too.
\end{proof}

In order to prove the invariant version of the Wedderburn--Artin theorem,
we need the following lemma:

  \begin{lemma}\label{LemmaAnnIdealTGradedGAction}
  Let $A=\bigoplus_{t\in T} A^{(t)}$
be an algebra over a field $\mathbbm{k}$ with a $T$-graded $G$-action 
for some groups $G$ and $T$.
Let $M\subseteq A$ be a $G$-invariant $T$-graded subspace.
Then $$\Ann_{\mathrm{lr}}(M) := \lbrace a \in A \mid ab=ba=0 \text{ for all } b \in M \rbrace$$ 
is a $G$-invariant $T$-graded subspace too.
  \end{lemma}
  \begin{proof}Being an annihilator of a  $T$-graded subspace,
  the space $\Ann_{\mathrm{lr}}(M)$ is $T$-graded.
  We claim that $\Ann_{\mathrm{lr}}(M)$ is, in addition, $G$-invariant.
  
  Let $a \in \Ann_{\mathrm{lr}}(M) \cap A^{(s)}$, $b\in M \cap A^{(t)}$, $s,t \in T$, $g\in G$. Then
$$ g^{-1}\bigl((ga)b\bigr)=
 \alpha(g^{-1},s,t)a(g^{-1}b)+\beta(g^{-1},s,t)(g^{-1}b)a
=0.$$
Hence $(ga)b=0$. 

Analogously,
$$g^{-1}(b(ga))= \alpha(g^{-1},t,s)(g^{-1}b)a+\beta(g^{-1},t,s)a(g^{-1}b)=0.$$
Hence $b(ga)=0$.

Since the equalities \begin{equation}\label{EqGabbga}(ga)b=b(ga)=0\end{equation} are linear in $a$ and $b$ and
the spaces $M$ and $\Ann_{\mathrm{lr}}(M)$
are graded, the
equalities~\eqref{EqGabbga} hold for all $a \in \Ann_{\mathrm{lr}}(M)$, $b\in M$.
    Therefore, $ga \in \Ann_{\mathrm{lr}}(M)$, which implies that $\Ann_{\mathrm{lr}}(M)$
is $G$-invariant.
\end{proof}

Now we are ready to prove the invariant version of the Wedderburn--Artin theorem:

\begin{theorem}\label{TheoremTGradedGActionInvWedderburnArtin} Let $\Gamma \colon B=\bigoplus_{t\in T} B^{(t)}$ be a $T$-grading on a finite dimensional semisimple associative algebra $B$ over a field $\mathbbm{k}$
with a $T$-graded $G$-action where $T$ and $G$ are arbitrary groups.
Then $$B = B_1 \oplus B_2 \oplus \dots \oplus B_s$$
  (direct sum of $T$-graded $G$-invariant ideals) for some algebras $B_i$
  with a graded $T$-action where $T$ and $G$ such that each $B_i$
  is simple in the appropriate sense, i.e. does not contain
nontrivial $T$-graded $G$-invariant ideals.
\end{theorem}
\begin{proof}
Since $B$ is finite dimensional, we can find a minimal $T$-graded $G$-invariant ideal $B_1$ in~$B$. By the original 
Wedderburn--Artin theorem, $B_1$ is a direct sum of some minimal ideals of $B$ and $B = B_1 \oplus \Ann_{\mathrm{lr}}(B_1)$
where $ \Ann_{\mathrm{lr}}(B_1)$ is the direct sum of the other minimal ideals of $B$. By Lemma~\ref{LemmaAnnIdealTGradedGAction}
the ideal  $ \Ann_{\mathrm{lr}}(B_1)$ is a $T$-graded $G$-invariant ideal and we may apply the induction argument to $ \Ann_{\mathrm{lr}}(B_1)$. 
\end{proof}

We conclude the section with another example of an algebra $A$ such that $\tau\left(\widetilde{\Aut}(A)\right)=Q:=\left\lbrace
\left(\begin{smallmatrix}
\alpha & \beta \\
\beta  & \alpha
\end{smallmatrix}\right)
\mathrel{\bigl|} \alpha^2-\beta^2 \ne 0 \right\rbrace$. (See Section~\ref{SectionUngradedPseudoAuto}.)

Consider the vector space $J=\langle u, v, w \rangle_\mathbbm{k}$ with the basis $u,v,w$
and the bilinear form $$\theta \colon J \times J \to \mathbbm{k}$$
with the matrix  $\left(\begin{array}{rrr}
1   & 1 & 0 \\
-1   & 0 & 0 \\
0   & 0 & 0
\end{array}\right)$ in this basis.
Define the product $ab := \theta(a,b)w$ for all $a,b\in J$. It is easy to see
that $J^2=\langle w \rangle_\mathbbm{k}$ and $J^3=0.$ In particular, $J$ is a nilpotent associative algebra.
Consider $A_0:= \mathbbm{k}1\oplus J$ where $1$ is the adjoint identity element of $A_0$.

\begin{theorem}\label{TheoremPseudoAutoImageTauIsQNilpAlg}
	Let $A_0$ be the algebra defined above, $\ch \mathbbm{k} \ne 2$. 
	Then $$\widetilde{\Aut}(A_0) = \Aut(A_0) \leftthreetimes \left(\mathbbm{k}^\times \times \mathbbm{k}^\times \right)$$
	and if we identify the linear operators $\varphi$ on $A_0$ with their matrices
	in the basis $1, u, v, w$,
	then 
	\begin{equation*}\begin{split}\widetilde{\Aut}(A_0) =\left\lbrace
	\left(\begin{array}{cccc}
	\frac{1}{\alpha+\beta} &  0   & 0 & 0 \\
	0 &  c_{22}   & 0 & 0 \\
	0 &  c_{32}   & c_{33} & 0 \\
	0 &  c_{42}   & c_{43} & c_{44}
	\end{array}\right) \right|  c_{22}, c_{33}, c_{44}\in \mathbbm{k}^\times,\\ \left. c_{32}, c_{42}, c_{43}\in \mathbbm{k},\ \alpha = \frac{c_{44}}{2c_{22}^2} \frac{c_{33}+c_{22}}{c_{33}},
	\  \beta = \frac{c_{44}}{2c_{22}^2} \frac{c_{33}-c_{22}}{c_{33}}
	\right\rbrace,\end{split}\end{equation*}
		$$\Aut(A_0) =\left\lbrace\left.
			\left(\begin{array}{cccc}
				1 &  0   & 0 & 0 \\
				0 &  c_{22}   & 0 & 0 \\
				0 &  c_{32}   & c_{22} & 0 \\
				0 &  c_{42}   & c_{43} & c_{22}^2
			\end{array}\right) \right|  c_{22}\in \mathbbm{k}^\times,\ c_{32}, c_{42}, c_{43}\in \mathbbm{k}\right\rbrace,$$
	$\tau(\varphi)=\left(\begin{smallmatrix}
\alpha & \beta \\
\beta  & \alpha
\end{smallmatrix}\right)$
	and $\mathbbm{k}^\times \times \mathbbm{k}^\times\cong Q$ is embedded in $\widetilde{\Aut}(A_0)$
	by $$\left(\begin{smallmatrix}
	\alpha & \beta \\
	\beta  & \alpha
	\end{smallmatrix}\right) \mapsto \left(\begin{array}{cccc}
	\frac{1}{\alpha+\beta} &  0   & 0 & 0 \\
	0 &  \alpha-\beta   & 0 & 0 \\
	0 &  0   & \alpha+\beta & 0 \\
	0 &  0   & 0 & (\alpha^2-\beta^2)(\alpha-\beta)	\end{array}\right).$$

	In particular, $\widetilde{\Aut}(A_0)$ is isomorphic to the group of all invertible lower triangular
	$3\times 3$ matrices.
\end{theorem}
\begin{proof}
	Let $\varphi \in \widetilde{\Aut}(A_0)$
	satisfy~\eqref{EqVarphiAB} for some $\alpha, \beta \in \mathbbm{k}$.
	Proposition~\ref{PropositionUnitPseudoAuto}
	implies $\varphi(1)=\frac{1}{\alpha+\beta}$.
	Since $J^3=0$, the ideal $J$ is the Jacobson radical of $A_0$.
	Hence by Theorem~\ref{TheoremTGradedGActionInvRad} we have
	$\varphi(J)=J$.
	Moreover, $\varphi(w)=\varphi(uv)\in \varphi(J)^2 = J^2= \langle w \rangle_\mathbbm{k}$.
	Denote by $C$ the matrix of $\varphi$
	in the basis $1, u, v, w$. 
	Then the remarks above imply that $$C=\left(\begin{array}{cccc}
	\frac{1}{\alpha+\beta} &  0   & 0 & 0 \\
	0 &  c_{22}   & c_{23} & 0 \\
	0 &  c_{32}   & c_{33} & 0 \\
	0 &  c_{42}   & c_{43} & c_{44}
	\end{array}\right).$$
	
	Let $C_0 =\left(\begin{array}{cc}
	c_{22}   & c_{23} \\
	c_{32}   & c_{33} \\
	\end{array}\right)$, $B_0 = \left(\begin{array}{rr}
	1   & 1 \\
	-1   & 0 \\
	\end{array}\right)$, $a=\xi_1 u + \xi_2 v$, $b=\eta_1 u + \eta_2 v$,
	$\boldsymbol\xi=\left(\begin{smallmatrix}
	\xi_1 \\
	\xi_2
	\end{smallmatrix}\right)$, $\boldsymbol\eta=\left(\begin{smallmatrix}
	\eta_1 \\
	\eta_2
	\end{smallmatrix}\right)$, $\xi_1, \xi_2, \eta_1, \eta_2 \in \mathbbm{k}$. Then~\eqref{EqVarphiAB}
	can be rewritten as 
	$$c_{44}\boldsymbol\xi^T B_0 \boldsymbol\eta=\alpha \boldsymbol\xi^T C_0^T B_0 C_0 \boldsymbol\eta
	+ \beta \boldsymbol\eta^T C_0^T B_0 C_0 \boldsymbol\xi,$$
	$$c_{44}\boldsymbol\xi^T B_0 \boldsymbol\eta=\alpha \boldsymbol\xi^T C_0^T B_0 C_0 \boldsymbol\eta
	+ \beta \boldsymbol\xi^T C_0^T B_0^T C_0 \boldsymbol\eta,$$
	$$\boldsymbol\xi^T \left(c_{44} B_0\right) \boldsymbol\eta=\boldsymbol\xi^T \left( C_0^T \left(\alpha B_0
	+  \beta B_0^T\right)  C_0\right) \boldsymbol\eta.$$
	
	Since the columns $\boldsymbol\xi$ and $\boldsymbol\eta$ are arbitrary,
	we get 
	$$c_{44} B_0 = C_0^T \left(\alpha B_0
	+  \beta B_0^T\right)  C_0,$$
	$$\left(\begin{array}{rc}
	c_{44}   & c_{44} \\
	-c_{44}   & 0 \\
	\end{array}\right) = C_0^T \left(\begin{array}{cc}
	\alpha+\beta   & \alpha-\beta \\
	\beta-\alpha   & 0 \\
	\end{array}\right)  C_0,$$
	$$\left(\begin{array}{rc}
	c_{44}   & c_{44} \\
	-c_{44}   & 0 \\
	\end{array}\right) = \left(\begin{array}{cc}
	(\alpha+\beta)c_{22}^2   & (\alpha+\beta)c_{22}c_{23} + (\alpha-\beta)\det C_0 \\
	(\alpha+\beta)c_{22}c_{23} +(\beta-\alpha)\det C_0   & (\alpha+\beta)c_{23}^2 \\
	\end{array}\right).$$
	Hence $c_{23}=0$, $c_{44}= (\alpha+\beta)c_{22}^2$ and $c_{44} = (\alpha-\beta)\det C_0 = (\alpha-\beta) c_{22}c_{33}$. Therefore, 
	$$C=\left(\begin{array}{cccc}
	\frac{1}{\alpha+\beta} &  0   & 0 & 0 \\
	0 &  c_{22}   & 0 & 0 \\
	0 &  c_{32}   & c_{33} & 0 \\
	0 &  c_{42}   & c_{43} & c_{44}
	\end{array}\right)$$ and
	$$\alpha = \frac{c_{44}}{2c_{22}^2} \frac{c_{33}+c_{22}}{c_{33}},\qquad
	\beta = \frac{c_{44}}{2c_{22}^2} \frac{c_{33}-c_{22}}{c_{33}}.$$
	Conversely, if $\alpha$, $\beta$ and $c_{ij}$ satisfy the equalities above,
	\eqref{EqVarphiAB} holds for all $a,b \in \langle u, v \rangle_\mathbbm{k}$.
	The fact that~\eqref{EqVarphiAB} holds if $a \in \lbrace 1, w \rbrace$
	or $b \in \lbrace 1, w \rbrace$ is checked trivially.
\end{proof}

\section{Graded polynomial $H$-identities}

If an algebra is endowed with a generalized $H$-action compatible with a $T$-grading, then one can consider the corresponding polynomial identities.

Let $T$ be a set. Recall that by $X^{T\text{-}\mathrm{gr}}$ we denote the disjoint union $\bigsqcup_{t \in T}X^{(t)}$ of sets $X^{(t)} = \{ x^{(t)}_1, x^{(t)}_2, \ldots \}$. The absolutely free non-associative algebra 
$\mathbbm{k}\left\lbrace X^{T\text{-}\mathrm{gr}} \right\rbrace$
 has a natural $\mathbb Z$-grading $\mathbbm{k}\left\lbrace X^{T\text{-}\mathrm{gr}} \right\rbrace =  \bigoplus\limits_{n=1}^\infty \mathbbm{k}\left\lbrace X^{T\text{-}\mathrm{gr}}\right\rbrace^{(n)}$ where $ \mathbbm{k}\left\lbrace X^{T\text{-}\mathrm{gr}}\right\rbrace^{(n)}$  is the linear span of all monomials of total degree $n$.
Let $H$ be an arbitrary associative algebra with $1$ over $\mathbbm{k}$.
Consider the algebra $$\mathbbm{k}\lbrace X^{T\text{-}\mathrm{gr}}|H \rbrace := \bigoplus_{n=1}^\infty H^{{}\otimes n} \otimes \mathbbm{k} \left\lbrace X^{T\text{-}\mathrm{gr}}\right\rbrace^{(n)}$$
with the multiplication $(u_1 \otimes w_1)(u_2 \otimes w_2):=(u_1 \otimes u_2) \otimes w_1w_2$
for all $u_1 \in  H^{{}\otimes j}$, $u_2 \in  H^{{}\otimes k}$,
$w_1 \in \mathbbm{k} \left\lbrace X^{T\text{-}\mathrm{gr}} \right\rbrace^{(j)}$, $w_2 \in \mathbbm{k} \left\lbrace X^{T\text{-}\mathrm{gr}} \right\rbrace^{(k)}$.
We use the notation $$\left(x_{i_1}^{(t_1)}\right)^{h_1}
\left(x_{i_2}^{(t_2)}\right)^{h_2}\ldots \left(x_{i_n}^{(t_n)}\right)^{h_n} := (h_1 \otimes h_2 \otimes \dots \otimes h_n) \otimes x_{i_1}^{(t_1)}
x_{i_2}^{(t_2)}\ldots x_{i_n}^{(t_n)}.$$ Here $h_1 \otimes h_2 \otimes \dots \otimes h_n \in H^{{}\otimes n}$,
$x_{i_1}^{(t_1)}
x_{i_2}^{(t_2)}\ldots x_{i_n}^{(t_n)} \in \mathbbm{k} \left\lbrace X^{T\text{-}\mathrm{gr}} \right\rbrace^{(n)}$. 
In addition, we identify $x_i^{(t)}$ with $\left(x_i^{(t)}\right)^{1_H}$ and treat $X^{T\text{-}\mathrm{gr}}$ as a subset of 
$\mathbbm{k} \lbrace X^{T\text{-}\mathrm{gr}}|H \rbrace$.

If $(\gamma_\beta)_{\beta \in \Lambda}$ is a basis in $H$, 
then $\mathbbm{k}\lbrace X^{T\text{-}\mathrm{gr}}|H \rbrace$ is isomorphic to the absolutely free non-associative algebra over $\mathbbm{k}$ with free formal  generators $\left(x^{(t)}_i\right)^{\gamma_\beta}$, $t\in T$, $\beta \in \Lambda$, $i \in \mathbb N$.
We refer to the elements
of $\mathbbm{k}\lbrace X^{T\text{-}\mathrm{gr}}|H \rbrace$ as \textit{$T$-graded $H$-polynomials}.

Let $A=\bigoplus_{t\in T} A^{(t)}$ be $T$-graded algebra with a compatible generalized $H$-action.
Suppose a map $\psi \colon X^{T\text{-}\mathrm{gr}} \to A$ satisfies the property $\psi\left(X^{(t)}\right) \subseteq A^{(t)}$ for every $t\in T$. Then $\psi$ has the unique homomorphic extension $\bar\psi
\colon \mathbbm{k} \lbrace X^{T\text{-}\mathrm{gr}}|H \rbrace \to A$ such that $\bar\psi\left(\left(x^{(t)}_i\right)^h\right)=h\psi\left(x_i^{(t)}\right)$
for all $i \in \mathbb N$ and $h \in H$.
We say that $f \in \mathbbm{k}\lbrace X^{T\text{-}\mathrm{gr}}|H \rbrace$
is a \textit{$T$-graded polynomial $H$-identity} of $A$ if $\bar\psi(f)=0$
for all such maps $\psi \colon X^{T\text{-}\mathrm{gr}} \to A$. In other words, $f=f\left(x^{(t_1)}_{i_1}, \ldots, x^{(t_s)}_{i_s}\right)$
is a $T$-graded  $H$-identity of $A$
if and only if $f\left(a^{(t_1)}_1, \ldots, a^{(t_s)}_s\right)=0$
for all $a^{(t_j)}_j \in A^{(t_j)}$, $1 \leqslant j \leqslant s$.
In this case we write $f \equiv 0$.
The set $\Id^{T\text{-}\mathrm{gr}, H}(A)$ of all $T$-graded polynomial $H$-identities
of $A$ is an ideal of $\mathbbm{k}\lbrace X^{T\text{-}\mathrm{gr}}|H \rbrace$.

\begin{remark}
	Note that here we do not consider any $H$-action and any grading on $\mathbbm{k} \lbrace X^{T\text{-}\mathrm{gr}}|H \rbrace$.
	However, extending the category of graded algebras with a generalized $H$-action
	analogously to what was done in~\cite[Section 7]{ASGordienko16}, we can make $\mathbbm{k} \lbrace -|H \rbrace$ a free functor corresponding to a free-forgetful adjunction.
	
\end{remark}

Denote by $W_n^{T\text{-}\mathrm{gr}, H}$ the space of all multilinear non-associative $T$-graded $H$-polynomials
in $x_1, \ldots, x_n$, $n\in\mathbb N$, i.e.
 $$W_n^{T\text{-}\mathrm{gr}, H} := \left\langle\left. \left(x_{\sigma(1)}^{(t_1)}\right)^{h_1} \left(x_{\sigma(2)}^{(t_2)}\right)^{h_2}\ldots \left(x_{\sigma(n)}^{(t_n)}\right)^{h_n} \,\right|\, \sigma \in S_n,\ t_i\in T,\ h_i\in H\right\rangle_\mathbbm{k} \subset \mathbbm{k}\lbrace X|H \rbrace.$$
(We consider all possible arrangements of brackets.)

The number $c_n^{T\text{-}\mathrm{gr}, H}(A):=\dim\left(\frac{W_n^{T\text{-}\mathrm{gr}, H}}{W_n^{T\text{-}\mathrm{gr}, H}\cap\, \Id^{T\text{-}\mathrm{gr}, H}(A)\strut}\right)$
is called the $n$th \textit{codimension of $T$-graded polynomial $H$-identities}
or the $n$th \textit{the $T$-graded $H$-codimension} of $A$.

The limit $\PIexp^{T\text{-}\mathrm{gr}, H}(A):=\lim\limits_{n\rightarrow\infty} \sqrt[n]{c^{T\text{-}\mathrm{gr}, H}_n(A)}$,
if it exists, is called \textit{the exponent of codimension growth of $T$-graded polynomial $H$-identities} or \textit{the $T$-graded $H$-PI-exponent} of~$A$.

One of the main tools in the investigation of polynomial
identities is provided by the representation theory of symmetric groups.
The symmetric group $S_n$  acts
on the space $\frac {W^{T\text{-}\mathrm{gr}, H}_n}{W^{T\text{-}\mathrm{gr}, H}_{n}
	\cap \Id^{T\text{-}\mathrm{gr}, H}(A)\strut}$
by permuting the variables inside each set $X^{(t)}$: $$\sigma \left(x^{(t_1)}_{i_1}\right)^{h_1}\ldots \left(x^{(t_n)}_{i_n}\right)^{h_n}
:= \left(x^{(t_1)}_{\sigma(i_1)}\right)^{h_1}\ldots \left(x^{(t_n)}_{\sigma(i_n)}\right)^{h_n}$$
for $n\in\mathbb N$, $\sigma \in S_n$, $1\leqslant i_k\leqslant n$, $1\leqslant k \leqslant n$.
If the characteristic of the base field $\mathbbm{k}$ is zero, 
then irreducible $\mathbbm{k}S_n$-modules are described by partitions
$\lambda=(\lambda_1, \dots, \lambda_s)\vdash n$ and their
Young diagrams $D_\lambda$.
The character $\chi^{T\text{-}\mathrm{gr}, H}_n(A)$ of the
$\mathbbm{k}S_n$-module $\frac {W^{T\text{-}\mathrm{gr}, H}_n}{W^{T\text{-}\mathrm{gr}, H}_n
	\cap \Id^{T\text{-}\mathrm{gr}, H}(A)\strut}$ is
called the $n$th
\textit{cocharacter} of $T$-graded polynomial $H$-identities of $A$.
We can rewrite it as
a sum $$\chi^{T\text{-}\mathrm{gr}, H}_n(A)=\sum_{\lambda \vdash n}
m(A, T\text{-}\mathrm{gr}, H, \lambda)\chi(\lambda)$$ of
irreducible characters $\chi(\lambda)$.
The number $\ell_n^{T\text{-}\mathrm{gr}, H}(A):=\sum\limits_{\lambda \vdash n}
m(A, T\text{-}\mathrm{gr}, H, \lambda)$ is called the $n$th
\textit{colength} of $T$-graded polynomial $H$-identities of $A$.
We refer the reader to~\cite{Bahturin, DrenKurs, ZaiGia}
for an account
of $S_n$-representations and their applications to polynomial
identities.

\begin{remark}\label{RemarkPolynomialHIdentities}
	If $A$ is an ungraded algebra with a generalized $H$-action, then we may treat $A$ as a $T$-graded algebra with a generalized $H$-action for $T=\lbrace 1 \rbrace$. $\lbrace 1 \rbrace$-graded polynomial $H$-identities of $A$ are called just \textit{polynomial $H$-identities}, $W^H_n:=W^{\lbrace 1 \rbrace\text{-}\mathrm{gr}, H}_n$, $c^H_n(A):=c^{\lbrace 1 \rbrace\text{-}\mathrm{gr}, H}_n(A)$,
	$\chi^H_n(A):=\chi^{\lbrace 1 \rbrace\text{-}\mathrm{gr}, H}_n(A)$, $m(A, H, \lambda):=m(A, \lbrace 1 \rbrace\text{-}\mathrm{gr}, H, \lambda)$, $\ell_n^H(A):=\ell_n^{\lbrace 1 \rbrace\text{-}\mathrm{gr}, H}(A)$.	
\end{remark}

\begin{remark}
Analogously to graded polynomial identities and polynomial $H$-identities one can introduce the Lie algebra  $L(
X^{T\text{-}\mathrm{gr}}|H)$ and the associative algebra $\mathbbm{k}\langle
 X^{T\text{-}\mathrm{gr}}|H\rangle$, define for a $T$-graded Lie (resp. associative) algebra with a compatible generalized $H$-action the corresponding $T$-graded polynomial $H$-identities and show
 that codimensions, colengths and cocharacters do not depend on whether  the $T$-graded
 $H$-polynomials are taken from $\mathbbm{k}\lbrace
 X^{T\text{-}\mathrm{gr}}|H\rbrace$ or $L(
 X^{T\text{-}\mathrm{gr}}|H)$ (resp. $\mathbbm{k}\langle
 X^{T\text{-}\mathrm{gr}}|H\rangle$). 
\end{remark}

In Theorem~\ref{TheoremGradGenActionReplace} 
it was shown that every generalized $H$-action compatible with a $T$-grading is reduced to a generalized $\mathbbm{k}^T\otimes H$-action. Below we show that the corresponding codimensions coincide.

\begin{proposition}\label{PropositionCnTGrHCnFTotimesH}
Let $\Gamma \colon A=\bigoplus_{t\in T} A^{(t)}$ be a $T$-grading on an algebra $A$ over a field $\mathbbm{k}$ where $T$ is a set.
Suppose $A$ is endowed with a generalized $H$-action of some associative algebra $H$ with $1$
such that $HA^{(t)}\subseteq A^{(t)}$ for all $t\in T$ and $\supp \Gamma$ is finite.
Then $$c_n^{T\text{-}\mathrm{gr}, H}(A)=c_n^{\mathbbm{k}^T\otimes H}(A)\text{\quad and\quad}\chi_n^{T\text{-}\mathrm{gr}, H}(A)=\chi_n^{\mathbbm{k}^T\otimes H}(A)\text{ for all }n\in \mathbb N.$$
 If, in addition, $\ch \mathbbm{k} = 0$, then $\ell_n^{T\text{-}\mathrm{gr}, H}(A)=\ell_n^{\mathbbm{k}^T\otimes H}(A)$.
 (Recall that $c_n^{\mathbbm{k}^T\otimes H}(A)$, $\chi_n^{\mathbbm{k}^T\otimes H}(A)$, $\ell_n^{\mathbbm{k}^T\otimes H}(A)$ were defined in Remark~\ref{RemarkPolynomialHIdentities} above.)
\end{proposition}
\begin{proof} 
Let $\xi \colon \mathbbm{k}\lbrace X | \mathbbm{k}^T \otimes H \rbrace \to \mathbbm{k}\lbrace X^{T\text{-}\mathrm{gr}} | H \rbrace$ be an algebra homomorphism defined by $\xi(x_i^{q\otimes h}) = \sum\limits_{t\in\supp \Gamma} q(t)\left(x^{(t)}_i\right)^h$ for all $i\in\mathbb N$, $q\in \mathbbm{k}^T$, $h\in H$. Let $f\in \Id^{\mathbbm{k}^T\otimes H}(A)$. Consider an arbitrary homomorphism $\psi \colon  
\mathbbm{k}\lbrace X^{T\text{-}\mathrm{gr}} | H \rbrace \to A$ such that
$\psi\left(x^{(t)}_i\right)\in A^{(t)}$ and $\psi\left( \left(x^{(t)}_i\right)^h\right)=h\psi\left(x^{(t)}_i\right)$
for all $t\in T$, $i\in\mathbb N$ and $h\in H$. Then for the algebra homomorphism $\psi\xi \colon \mathbbm{k}\lbrace X | \mathbbm{k}^T \otimes H \rbrace
\to A$ we have $$\psi\xi(x_i^{q\otimes h})=\sum\limits_{t\in\supp \Gamma} q(t)\psi\left(\left(x^{(t)}_i\right)^h\right)=
(q\otimes h)\sum\limits_{t\in\supp \Gamma} \psi\left(x^{(t)}_i\right)=(q\otimes h)\,\psi\xi(x_i).$$
Thus for every such homomorphism $\psi$ the equality $\psi\xi(f) =0$ holds. Hence
 $\xi(f)\in \Id^{T\text{-}\mathrm{gr}, H}(A)$ and $\xi\left(\Id^{\mathbbm{k}^T\otimes H}(A)\right)\subseteq \Id^{T\text{-}\mathrm{gr}, H}(A)$.
Denote by $$\tilde \xi \colon \mathbbm{k}\lbrace X | \mathbbm{k}^T \otimes H \rbrace/\Id^{\mathbbm{k}^T\otimes H}(A) \to \mathbbm{k}\lbrace X^{T\text{-}\mathrm{gr}} | H \rbrace/\Id^{T\text{-}\mathrm{gr}, H}(A)$$ the homomorphism induced by $\xi$.

Let $\eta \colon  \mathbbm{k}\lbrace X^{T\text{-}\mathrm{gr}} | H \rbrace \mathrel{\to} \mathbbm{k}\lbrace X | \mathbbm{k}^T \otimes H \rbrace$
 be the algebra homomorphism defined by $\eta\left(\left(x^{(t)}_i\right)^h\right) = x_i^{q_t \otimes h}$ for $i\in \mathbb N$
and $t\in T$. 
Consider an arbitrary $T$-graded polynomial $H$-identity $f\in \Id^{T\text{-}\mathrm{gr}, H}(A)$.
Let $\psi \colon  \mathbbm{k}\lbrace X | \mathbbm{k}^T \otimes H \rbrace \to A$ be the homomorphism
satisfying $\psi(x_i^{q\otimes h})=(q\otimes h)\psi(x_i)$ for all $i\in\mathbb N$, $q\in \mathbbm{k}^T$ and $h\in H$.
Then for all $i\in\mathbb N$ and $g, t \in T$ we have
$$q_g \psi\eta\left(x^{(t)}_i\right) = q_g\psi\left(x^{q_t\otimes 1}_i\right)=q_g q_t \psi(x_i)
=\left\lbrace \begin{array}{lll} 0 & \text{ for } & g\ne t,\\
                              \psi\eta\left(x^{(t)}_i\right) & \text{ for } & g=t. \end{array}\right.$$
 Hence $\psi\eta\left(x^{(t)}_i\right) \in A^{(t)}$.
 Moreover, $$ \psi\eta\left(\left(x^{(t)}_i\right)^h\right) = 
 \psi\left(x_i^{q_t\otimes h}\right)=h\,\psi\left(x_i^{q_t\otimes 1}\right)= h\, \psi\eta\left(x^{(t)}_i\right).$$
 Thus $\psi\eta(f)=0$ and $\eta(\Id^{T\text{-}\mathrm{gr}, H}(A)) \subseteq \Id^{\mathbbm{k}^T\otimes H}(A)$.
Denote by $$\tilde\eta \colon  \mathbbm{k}\lbrace X^{T\text{-}\mathrm{gr}} | H \rbrace/\Id^{T\text{-}\mathrm{gr}, H}(A) \to
\mathbbm{k}\lbrace X | \mathbbm{k}^T \otimes H \rbrace/\Id^{\mathbbm{k}^T\otimes H}(A)$$ the induced homomorphism.

Below we use the notation $\bar f = f + \Id^{\mathbbm{k}^T\otimes H}(A) \in \mathbbm{k}\lbrace X | \mathbbm{k}^T \otimes H \rbrace/\Id^{\mathbbm{k}^T\otimes H}(A)$ for $f\in
\mathbbm{k}\lbrace X | \mathbbm{k}^T \otimes H \rbrace$ and  $\bar f = f + \Id^{T\text{-}\mathrm{gr}, H}(A) \in \mathbbm{k}\lbrace X^{T\text{-}\mathrm{gr}} | H \rbrace/\Id^{T\text{-}\mathrm{gr}, H}(A)$ for $f\in \mathbbm{k}\lbrace X^{T\text{-}\mathrm{gr}} | H \rbrace$.
Note that $$x^{q\otimes h}_i - \sum\limits_{t\in\supp \Gamma} q(t) x^{q_t\otimes h}_i\in \Id^{\mathbbm{k}^T\otimes H}(A)$$ for all $q\in \mathbbm{k}^T$, $h\in H$ and $i\in\mathbb N$.
Hence $$\tilde\eta\tilde\xi\left(\overline{x^{q\otimes h}_i}\right)=\tilde\eta\left(
\sum\limits_{t\in\supp \Gamma} q(t) \overline{\left(x^{(t)}_i\right)^h}\right)
=\sum\limits_{t\in\supp \Gamma} q(t) \overline{x^{q_t\otimes h}_i} = \overline{x^{q\otimes h}_i}$$
for all $q\in \mathbbm{k}^T$, $h\in H$ and $i\in\mathbb N$. 
Therefore, $$\tilde\eta\tilde\xi=\id_{\mathbbm{k}\lbrace X | \mathbbm{k}^T \otimes H \rbrace/\Id^{\mathbbm{k}^T\otimes H}(A)}$$
since the algebra $\mathbbm{k}\lbrace X | \mathbbm{k}^T \otimes H \rbrace/\Id^{\mathbbm{k}^T\otimes H}(A)$ is generated by $\overline{x^{q\otimes h}_i}$ where $q\in \mathbbm{k}^T$, $h\in H$ and $i\in\mathbb N$.
Furthermore, $\tilde\xi\tilde\eta\Biggl(\overline{\left(x^{(t)}_i\right)^h}\Biggr)=
\tilde\xi\left(\overline{x^{q_t\otimes h}_i}\right)=\overline{\left(x^{(t)}_i\right)^h}$ for all $t\in \supp \Gamma$, $h\in H$ and $i\in \mathbb N$.
Hence $\tilde\xi\tilde\eta=\id_{\mathbbm{k}\lbrace X^{T\text{-}\mathrm{gr}} | H \rbrace/\Id^{T\text{-}\mathrm{gr}, H}(A)}$
and $$\mathbbm{k}\lbrace X^{T\text{-}\mathrm{gr}} | H \rbrace/\Id^{T\text{-}\mathrm{gr}, H}(A) \cong \mathbbm{k}\lbrace X | \mathbbm{k}^T \otimes H \rbrace/\Id^{\mathbbm{k}^T\otimes H}(A)$$
as algebras. The restriction of $\tilde\xi$ is an isomorphism of $\mathbbm{k}S_n$-modules
$\frac{W^{\mathbbm{k}^T\otimes H}_n}{W^{\mathbbm{k}^T\otimes H}_n \cap \Id^{\mathbbm{k}^T\otimes H}(A)}$ and $\frac{W^{T\text{-}\mathrm{gr}, H}_n}{W^{T\text{-}\mathrm{gr}, H}_n\cap \Id^{T\text{-}\mathrm{gr}, H}(A)}$.  Thus $$c^{\mathbbm{k}^T\otimes H}_n(A)=\dim \frac{W^{\mathbbm{k}^T\otimes H}_n}{W^{\mathbbm{k}^T\otimes H}_n \cap \Id^{\mathbbm{k}^T\otimes H}(A)}
= \dim\frac{W^{T\text{-}\mathrm{gr}, H}_n}{W^{T\text{-}\mathrm{gr}, H}_n\cap \Id^{T\text{-}\mathrm{gr}, H}(A)}=c^{T\text{-}\mathrm{gr}, H}_n(A)$$
and  $\chi_n^{T\text{-}\mathrm{gr}, H}(A)=\chi_n^{\mathbbm{k}^T\otimes H}(A)$
for all $n\in \mathbb N$.
 If, in addition, $\ch \mathbbm{k} = 0$, the above implies $\ell_n^{T\text{-}\mathrm{gr}, H}(A)=\ell_n^{\mathbbm{k}^T\otimes H}(A)$ too.
\end{proof}

If an algebra $A$ graded by a set $T$ is endowed with a $T$-graded action of a group $G$,
then $A$ is an algebra with a generalized $\mathbbm{k}G$-action compatible with the $T$-grading.
Therefore we may call the corresponding $T$-graded polynomial $\mathbbm{k}G$-identities of $A$
its \textit{$T$-graded polynomial $G$-identities} and define $\Id^{T\text{-}\mathrm{gr}, G}(A):=\Id^{T\text{-}\mathrm{gr}, \mathbbm{k}G}(A)$ and $c_n^{T\text{-}\mathrm{gr}, G}(A):=c_n^{T\text{-}\mathrm{gr}, \mathbbm{k}G}(A)$.

\begin{example}\label{ExampleGrGId}
	Consider the $\mathbb{Z}/2\mathbb{Z}$-grading $\UT_2(\mathbbm{k})=\UT_2(\mathbbm{k})^{(0)}\oplus \UT_2(\mathbbm{k})^{(1)}$
	on the algebra $\UT_2(\mathbbm{k})$ of upper triangular $2\times 2$ matrices over a field $\mathbbm{k}$
	from Example~\ref{ExampleIdGr}. Define a graded $\mathbbm{k}^\times$-action on $\UT_2(\mathbbm{k})$
	by $\alpha \cdot e_{12} := \alpha e_{12}$ and $\alpha \cdot e_{ii} := e_{ii}$ for $i=1,2$.
	Then $$\left( x^{(0)} \right)^\alpha - x^{(0)},\ \left( x^{(1)} \right)^\alpha - \alpha x^{(1)}\in \Id^{\mathbb{Z}/2\mathbb{Z} \text{-}\mathrm{gr}, \mathbbm{k}^\times}(A)\text{ for all }\alpha\in \mathbbm{k}^\times.$$
\end{example}

\begin{theorem}\label{TheoremTGrGActionAssoc}
Let $A$ be a finite dimensional associative algebra over a field $\mathbbm{k}$ of characteristic $0$
graded by an arbitrary group $T$ and endowed with a $T$-graded action of a group $G$.
Then \begin{enumerate}
\item either there exists $n_0$ such that $c_n^{T\text{-}\mathrm{gr}, G}(A)=0$ for $n\geqslant n_0$;
\item or there exist constants $C_1, C_2 > 0$, $r_1, r_2 \in \mathbb R$,
  $d \in \mathbb N$ such that $$C_1 n^{r_1} d^n \leqslant c^{T\text{-}\mathrm{gr}, G}_n(A)
   \leqslant C_2 n^{r_2} d^n\text{ for all }n \in \mathbb N.$$
\end{enumerate}
In particular, for every such $A$ there exists $\PIexp^{T\text{-}\mathrm{gr},G}(A):=\lim\limits_{n\to\infty} \sqrt[n]{c^{T\text{-}\mathrm{gr},G}_n(A)}\in\mathbb Z_+$
and the analog of Amitsur's holds. 
\end{theorem}
\begin{proof} As it was shown in Theorem~\ref{TheoremGradGenActionReplace},
$A$ is an algebra with a generalized $\mathbbm{k}^T\otimes \mathbbm{k}G$-action.
Moreover, by Proposition~\ref{PropositionCnTGrHCnFTotimesH}
we have $c_n^{T\text{-}\mathrm{gr}, G}(A)=c_n^{\mathbbm{k}^T\otimes \mathbbm{k}G}(A)$
for all $n\in\mathbb N$. As usual (the proof is analogous to \cite[Theorem~4.1.9]{ZaiGia}),
the codimensions $c_n^{\mathbbm{k}^T\otimes \mathbbm{k}G}(A)$
do not change upon an extension of the base field $\mathbbm{k}$.
Hence without loss of generality we may assume $\mathbbm{k}$ to be algebraically closed.
By Theorem~\ref{TheoremTGradedGActionInvRad}
the radical $J(A)$ is a  $\mathbbm{k}^T\otimes \mathbbm{k}G$-submodule
and by Theorem~\ref{TheoremTGradedGActionInvWedderburnArtin}
the algebra $A/J(A)$ is a direct sum of $\mathbbm{k}^T\otimes \mathbbm{k}G$-invariant
ideals $B_i$ that are $\mathbbm{k}^T\otimes \mathbbm{k}G$-simple algebras.
Now Theorem~\ref{TheoremTGrGActionAssoc} follows from~\cite[Theorem~1]{ASGordienko8}.
\end{proof}

\section*{Acknowledgments}

The author is grateful to the anonymous referee for carefully reading the manuscript and making useful remarks.

\end{document}